\documentclass[11pt, reqno]{amsart}

\usepackage[utf8]{inputenc}
\usepackage[margin = 1in]{geometry}
\usepackage{amsmath}
\usepackage{amsthm}
\usepackage{amssymb}
\usepackage{natbib}
\usepackage{xcolor}
\usepackage{algorithm}
\usepackage{algpseudocode}
\usepackage{comment}
\usepackage{cleveref}
\usepackage{appendix}
\newtheorem{theorem}{Theorem}[section]
\newtheorem{lemma}{Lemma}[section]
\newtheorem{remark}{Remark}[section]
\newtheorem{proposition}{Proposition}[section] 
\newtheorem{corollary}{Corollary}[section]
\newtheorem{assumption}{Assumption}[section]

\newtheorem{definition}{Definition}[section]
\begin{document}
\title{Foundations of locally-balanced Markov processes}
\author[Samuel Livingstone, Giorgos Vasdekis and Giacomo Zanella]{}
% \author{Andrea Bertazzi\textsuperscript{1}}
% \thanks{\textsuperscript{1} CMAP, Ecole polytechnique, Paris, France. Email: \texttt{andrea.bertazzi@polytechnique.edu}}

% \author{Giorgos Vasdekis\textsuperscript{2}}
% \thanks{\textsuperscript{2} School of Mathematics, Statistics and Physics, Newcastle University, Newcastle upon Tyne, U.K. Email: \texttt{giorgos.vasdekis@newcastle.ac.uk}}

\date{}

\maketitle
\vspace{-10pt}
\begin{center}
    \textsc{Samuel Livingstone} \\
    Department of Statistical Science, University College London,  U.K.\\
    \texttt{samuel.livingstone@ucl.ac.uk} \\
    \bigskip
    \textsc{Giorgos Vasdekis} \\
	School of Mathematics, Statistics and Physics, Newcastle University, U.K. \\ \texttt{giorgos.vasdekis@newcastle.ac.uk}\\
    \bigskip
    \textsc{Giacomo Zanella} \\
	Department of Decision Sciences and BIDSA, Bocconi University, Italy \\
    \texttt{giacomo.zanella@unibocconi.it}
\end{center}

\begin{abstract}
    We formally introduce and study locally-balanced Markov jump processes (LBMJPs) defined on a general state space. These continuous-time stochastic processes with a user-specified limiting distribution are designed for sampling in settings involving discrete parameters and/or non-smooth distributions, addressing limitations of other processes such as the overdamped Langevin diffusion. The paper establishes the well-posedness, non-explosivity, and ergodicity of LBMJPs under mild conditions. We further explore regularity properties such as the Feller property and characterise the weak generator of the process. We then derive conditions for exponential ergodicity via spectral gaps and establish comparison theorems for different balancing functions. In particular we show an equivalence between the spectral gaps of Metropolis--Hastings algorithms and LBMJPs with bounded balancing function, but show that LBMJPs can exhibit uniform ergodicity on unbounded state spaces when the balancing function is unbounded, even when the limiting distribution is not sub-Gaussian. We also establish a diffusion limit for an LBMJP in the small jump limit, and discuss applications to Monte Carlo sampling and non-reversible extensions of the processes.
\end{abstract}


\section{Introduction}

Continuous-time stochastic processes with a user-specified limiting distribution are the backbone of modern sampling algorithms, making them an indispensable tool for modern scientific research.  Celebrated examples include the overdamped Langevin diffusion \cite{roberts2002langevin,xifara2014langevin}, its kinetic/underdamped counterpart \cite{dalalyan2020sampling}, randomized Hamiltonian Monte Carlo \cite{bou2017randomized}, and piecewise-deterministic Markov processes such as the zigzag and bouncy particle samplers \cite{bierkens2019zig,bouchard2018bouncy}.  These processes are now well understood, mixing properties have been meticulously studied (e.g.\ \cite{roberts1996exponential,mattingly2002ergodicity,bou2017randomized,bierkens.roberts.zitt:19,vasdekis2022note,durmus.guillin.monmarche:21,deligiannidis2019exponential}), and the resulting sampling algorithms have been successfully used in many applied domains (e.g.\ \cite{diggle2019modeling,rossky1978brownian,duane1987hybrid,betancourt2015hamiltonian,hardcastle2024averaging,koskela2022zig}).

An important component of the above processes is their use of gradients to drive sample paths into regions of high probability under the limiting distribution. The overdamped Langevin diffusion, for example, is governed by the equation $dX_t = -(1/2)\nabla U(X_t)dt + dB_t$ and has a limiting distribution with density $\pi(x) \propto e^{-U(x)}$.  To simulate the process, the potential $U$ must therefore be sufficiently smooth. This restricts applicability to smooth manifolds (e.g.\ \cite{betancourt2017geometric,livingstone2014information}).  Many problems requiring sampling algorithms are defined on more general spaces involving discrete parameters and/or non-smooth distributions of interest.  There is, therefore, a need to develop and study processes that can be applied in these settings.

Locally-balanced Markov processes are a relatively new class of pure jump-type stochastic processes that have been designed for this purpose.  Based on the locally informed algorithm of \cite{zanella2020informed}, the continuous-time formulation in finite state spaces was introduced by \cite{power2019accelerated}.  Many approaches for both discrete and continuous sampling problems have since been motivated by these processes, such as discrete state-space algorithms \cite{zhou2022dimension,chang2024dimension,grathwohl2021oops,zhang2022langevin,liang2023adaptive,van2022g,liang2023structure,liang2022adaptive,sansone2022lsb,sun2022optimal,sun2023any}, multiple try algorithms \cite{gagnon2023improving,chang2022rapidly}, gradient-based algorithms \cite{livingstone.zanella:22,vogrinc.livingstone.zanella:22,mauri.zanella:24,hird.livingstone.zanella:22}, reversible jump algorithms \cite{gagnon2021informed}, use as a Stein operator \cite{shi2022gradient} and importance sampling of Markov chain sample paths \cite{zhou2022rapid,li2023importance}.  

In this article we formally introduce and study the locally-balanced Markov process defined on a general state space.  In \Cref{section.def.LBMJP:00} we define the process and in \Cref{sec:wellposed} we show that it is well-defined, non-explosive and has a user-defined invariant distribution, and that it is ergodic under mild conditions. In \Cref{sec::Feller} we prove regularity properties, such as the Feller property, while in \Cref{sec::Generator} we study the generator of the process. In \Cref{sec::Mixing} we focus on mixing properties. In particular, in \Cref{sec:expo.ergo:000} we provide conditions under which the process converges exponentially quickly to equilibrium, and provide comparison theorems to compare the spectral gaps of different locally-balanced processes.  In \Cref{subsec:unif_erg} we show that in some cases the process is uniformly ergodic on unbounded state spaces, even when the limiting distribution is not sub-Gaussian. We note that this is qualitatively different behaviour to the overdamped Langevin diffusion, which is typically not uniformly ergodic when the limiting distribution is not sub-Gaussian.  In \Cref{convergence.diffusion:000} we consider a particular regime in which we show that the process has a diffusion limit. Finally, in \Cref{sec::Monte Carlo} we discuss using the locally-balanced Markov process for Monte Carlo sampling, and in \Cref{sec::Non-reversible} we suggest nonreversible extensions of the process, a topic of recent interest among the sampling community (e.g.\ \cite{diaconis.holmes.neal:00,andrieu2021peskun,faulkner2024sampling}). The appendices contain some of the longer and more technical proofs.

\subsection{Notation}
Let $E$ be a Polish space. For any measure $\pi$ defined on the measurable space $(E,\mathcal{E})$, if $\pi$ admits a density with respect to some dominating measure, we will write $\pi$ to denote both the measure and the density when there is no ambiguity. We write $\textup{supp}(\pi)$ to indicate the support of $\pi$. Furthermore, for a Markov kernel $\gamma:E\ \times \mathcal{E} \rightarrow [0,1]$, we will sometimes write $[\pi \otimes \gamma]$ to denote the measure on $E \times E$ defined as $[\pi \otimes \gamma](dx,dy) := \pi(dx)\gamma(x,dy)$.

We denote the indicator function of the set $A$ as $\mathbb{I}_A$. We write $L^1(\pi), L^2(\pi)$ and $\mathcal{B}(E)$ to denote  the set of functions $f: E \rightarrow \mathbb{R}$ that are respectively integrable with respect to $\pi$, square-integrable with respect to $\pi$ and bounded. For $f,g \in L^2(\pi)$ we write $\langle f,g \rangle_\pi:=\int_Ef(x)g(x)\pi(dx)$. We write $C_b(E)$ and $C_0(E)$ to denote the continuous real functions $f$ with domain $E$ that are bounded or converging to zero at the boundary of $E$ respectively. Finally, in the case where $E=\mathbb{R}^d$ we write $C^{\infty}$ for the set of functions $f: \mathbb{R}^d \rightarrow \mathbb{R}$ that are infinitely differentiable and $C^{\infty}_c$ for those that are infinitely differentiable with compact support.  The composition $f \circ h(x) := f(h(x))$ for any two functions $f,h$ such that the range of $h$ is within the domain of $f$

\section{Definition and basic properties}

\subsection{Definition of the process}\label{section.def.LBMJP:00}

Consider a Borel measure space $(E,\mathcal{E})$ where $E$ is Polish. Consider also an underlying probability space $(\Omega, \mathcal{F}, \mathbb{P})$. Let $\pi$ be a measure on $(E,\mathcal{E})$ with $\pi(E) = 1$, and $\gamma:E\times \mathcal{E} \to [0,1]$ be a Markov transition kernel on $(E,\mathcal{E})$. We denote by $t(x,y)$ the Radon--Nikodym derivative
\begin{equation} \label{eq:radon_nikodym}
t(x,y) := \frac{\pi(dy)\gamma(y,dx)}{\pi(dx)\gamma(x,dy)} = \frac{[\pi \otimes \gamma](dy,dx)}{[\pi \otimes \gamma](dx,dy)},
\end{equation}
and note that a version of $t(x,y)$ is always well-defined and satisfies $0 < t(x,y) < \infty$ and $t(y,x) = 1/t(x,y)$ for all $(x,y) \in R$ for some $R \subset E \times E$ satisfying $[\pi \otimes \gamma](R) = 1$, and $t(x,y) := 0$ for $(x,y) \notin R$, by Proposition 1 of \cite{tierney1998note}. 
%Throughout we make the following standing assumption.

Before defining a locally-balanced Markov jump process in Definition \ref{def:LBMJP}, we must first introduce the notion of a balancing function.

\begin{definition}[Balancing function]
    We call $g:\mathbb{R}_{\geq 0} \to \mathbb{R}_{\geq 0}$ a \emph{balancing function} if $g(1) = 1$ and $g$ satisfies the balancing property
\begin{equation} \label{eq:balancing}
g(t) = tg(1/t)
\end{equation}
for all $t >0$, with the conventions that $g(0) := 0$ and $g(t)>0$ for $t>0$.
\end{definition} 

\begin{remark} 
The requirement $g(1)=1$ is essentially without loss of generality, as a different choice of balancing function $\tilde{g} = c\cdot g$ such that $\tilde{g}(1) = c$ for some $c > 0$ will induce a process with the same law at time $t$ as the original process at time $ct$, for all $t >0$. 
\end{remark}

\begin{remark}
There are an infinite number of balancing functions. For example, there is a bijection between the space of balancing functions and the space $\left\{ f: [0,1) \rightarrow \mathbb{R}_{> 0}   \right\}$, since one can always set 
\begin{equation*}
    g(t) =
\begin{cases}
  f(t) & \text{if } t \in (0,1)\\
  1 & \text{if } t=1 \\
  t f(1/t) & \text{if } t >1.
\end{cases}
\end{equation*}
Furthermore, it is easy to see that the space of balancing functions is convex with respect to addition and multiplication, i.e.\ for any $g_1, g_2$ balancing and $a \in (0,1)$, $a g_1+(1-a) g_2$ and $g_1^a g_2^{1-a}$ is also balancing. Finally, for any even function $h: \mathbb{R} \rightarrow \mathbb{R}_{>0}$, the function $g_h(t)=\sqrt{t}h(\log t)$ is also balancing (see e.g.\ Proposition 1 of \cite{vogrinc.livingstone.zanella:22}).  
\end{remark}

\begin{remark}
Popular choices of $g$ include the bounded functions $g(t)=\min(1 ,t)$ and $g(t)=2t/\left( 1+t  \right)$, as well as the unbounded choices $g(t)=\max(1, t)$ and $g(t) = t^\alpha\mathbb{I}_{[0,1)}(t) + t^{1-\alpha}\mathbb{I}_{[1,\infty)}(t)$ for any $a >0$ (e.g.\ setting $a=1/2$ gives $g(t) = \surd t$).
\end{remark}

Given $g$ and the Radon--Nikodym derivative $t(x,y)$ we can define the function
\begin{equation} \label{eq:lambda}
\lambda(x) := \int_E g \circ t(x,y)\gamma(x,dy)
\end{equation}
%to be the jump intensity from any state $x \in E$.  We will also write 
and the operator
\begin{equation}\label{def.gamma.g:1}
\Gamma_g f := \int_E f(y) \frac{g \circ t(x,y)}{\lambda(x)} \gamma(x,dy)
\end{equation}
for any $f \in L^2(\pi)$.  We use $\Gamma_g$ to denote both the operator and the corresponding Markov kernel $\Gamma_g(x,A) := \Gamma_g \mathbb{I}_A(x)$ for any $x \in E$ and $A \in \mathcal{E}$ when there is no ambiguity.  
We will often make the following assumption.
\begin{assumption}\label{lambda.positivity:00}
  For all $x \in \textup{supp}(\pi)$,  $\lambda(x) > 0$.
\end{assumption}

With these ingredients we are now able to define the process.

%With these ingredients we are now able to define a locally-balanced Markov jump process (LBMJP) with balancing function $g$, base kernel $\gamma$, target distribution $\pi$ and state space $E \cup \partial$, where $\partial$ is called a graveyard state. 

\begin{definition}[Locally-balanced Markov jump process]\label{def:LBMJP}
Let $x \in \textup{supp}(\pi)$, and consider a discrete time Markov Chain $X=(X_n)_{n \in \mathbb{N}}$, with $X_0=x$ and transition kernel $\Gamma_g$ as in (\ref{def.gamma.g:1}). Consider also a sequence of conditionally independent given $X$ random variables 
\begin{equation*}
    \tau_n \sim \textup{Exp}\left( \lambda(X_{n-1}) \right),
\end{equation*}
for $n \geq 1$, with $\lambda$ as in \eqref{eq:lambda}. Set $T_0=0$ and
\begin{equation*}
    T_n :=\sum_{k=1}^n \tau_k
\end{equation*}
for all $n \in \mathbb{N}$. Given $(X_n)_{n \geq 0}$ and $(T_n)_{n \geq 0}$, the \emph{Locally-Balanced Markov Jump Process (LBMJP)} with balancing function $g$, base kernel $\gamma$, target distribution $\pi$ and state space $E \cup \partial$ is $\left(Y_t\right)_{t \geq 0}$ such that $Y_0 = x$ and 
\begin{equation*}
    Y_t = 
    \begin{cases}
        X_{\sup\{n\,:\,T_n \leq t\}} \ ,  & \text{if } t < \sup\{T_n \,:\, n \in \mathbb{N} \}
        \\
        \partial \ , & \text{otherwise}.
    \end{cases}
\end{equation*}
We refer to $\partial$ as the graveyard state.
\end{definition}

Note that the above construction is problematic if $\lambda$ takes infinite values in a set of positive Lebesgue measure. In \Cref{sec:wellposed}, however, we provide conditions on $\lambda$, $g$, $\gamma$, and $\pi$ under which the process is well-posed and non-explosive, meaning that $\lambda$ is $\pi$-almost surely (a.s.) finite and the process will never reach the graveyard state $\partial$. 

The jump kernel associated to $(Y_t)_{t \geq 0}$ is defined for any $x \in E$ and any $A \in \mathcal{E}$ as $J(x,A) := \lambda(x)\Gamma_g(x,A)$. This can also be written
\begin{equation}
    J(x,dy) = g \circ t(x,y)\gamma(x,dy),
\end{equation}
from which the below proposition is straightforward.

\begin{proposition} \label{prop:reversible_jumpkernel}
    The jump kernel $J$ associated with a LBMJP $(Y_t)_{t \geq 0}$ satisfies 
    $$
    \int_{A\times B}\pi(dx)J(x,dy)=\int_{A \times B}\pi(dy)J(y,dx)
    $$
    for any $A,B \in \mathcal{E}$.
\end{proposition}

\begin{proof}
Note first that $\pi(dx)J(x,dy) = g \circ t(x,y)[\pi \otimes \gamma](dx,dy)$. The key step of the argument is that
\begin{align}\label{eq:symm_measure}
g \circ t(x,y)[\pi \otimes \gamma](dx,dy) = g \circ t(y,x) [\pi \otimes \gamma](dy,dx),
\end{align}    
which is simply a rearrangement of \eqref{eq:balancing} when $t(x,y)$ is chosen as in \eqref{eq:radon_nikodym}. The right-hand side of \eqref{eq:symm_measure} is equal to $\pi(dy)J(y,dx)$, which concludes the proof.
\end{proof}

A simple method to simulate the process is given in Algorithm \ref{alg:basic}.
\begin{algorithm}
\caption{Simulation of a locally-balanced Markov jump process}\label{alg:basic}
\begin{algorithmic}
\Require $x \in E$, $T^*>0$
\State Set $i \gets 0$, $T_0 \gets 0$, $X_0 \gets x$
\While{$T_i < T^*$}
\State Draw $\tau_{i+1} \sim \textup{Exp}(\lambda(X_i))$
\State Set $T_{i+1} \gets T_i + \tau_{i+1}$
\State Set $Y_t \gets X_i$ for all $t \in [T_i, T_{i+1})$.
\State Draw $X_{i+1} \sim \Gamma_g(X_i,\cdot)$
\State $i \gets i+1$.
\EndWhile
\end{algorithmic}
\Return $(Y_t)_{0 \leq t \leq T^*}$.
\end{algorithm}

\subsection{Well-posedness, non-explosivity and ergodicity} \label{sec:wellposed}

%In this section we study the stability and ergodicity of LBMJPs, and in particular provide conditions that guarantee that a LBMJP is well-defined (meaning that $\lambda(x)$ is finite), that it does not explode (meaning that it never reaches the graveyard state) and that it is ergodic.
%A key message of this section is that restricting to non-decreasing balancing functions $g$, as stated in the following assumption, is enough to guarantee most of the above properties. 

The below non-decreasing assumption on $g$ will prove sufficient to establish a number of fundamental properties for a locally-balanced Markov jump process.

\begin{assumption}\label{ass:g_non_decr}
$g:\mathbb{R}_{\geq 0} \rightarrow \mathbb{R}_{\geq 0}$ is non-decreasing and continuous.
\end{assumption}

\begin{remark}
    Assuming that $g$ is non-decreasing is arguably very natural when the kernel $\gamma$ is symmetric with respect to some canonical reference measure $\mu$ (e.g.\ Lebesgue or counting), in which case $t(x,y)=f(y)/f(x)$ with $f=d\pi/d\mu$, and $g$ gives higher weight to areas with higher density $f$. More generally, however, any choice of locally balancing function will have the correct stationary measure.  We leave exploration of other choices for future work.
\end{remark}

The following lemma details some constraints on $g$ imposed by \Cref{ass:g_non_decr}.
\begin{lemma}\label{lem:gbound1}
Let $g$ be a balancing function satisfying \Cref{ass:g_non_decr}. Then
\begin{enumerate}
    \item[(i)]
$g(t) \leq 1+t $ for all $t \geq 1$.
\item[(ii)] For all $t>0$,
\begin{equation}\label{eq:gbound1}
\min(1,t)
\leq
g(t)
\leq
\max(1,t).
\end{equation}
\end{enumerate}

\end{lemma}
\begin{proof}
(i) Since $g(1) = 1$ then by \Cref{ass:g_non_decr} $g(t)\leq 1$ for $t \in [0,1]$, from which direct calculation shows that $g(t) = tg(1/t) \leq t$ for $t \geq 1$. Combining gives the result.  (ii) For $t>1$ then $g(t) \geq 1$ follows immediately from the non-decreasing assumption, and for $t<1$ direct calculation gives $g(t) = tg(1/t) \geq t$, from which it follows that $g(t) \geq \min(1,t)$.  Switching the direction of the inequalities in each case gives the upper bound $g(t) \leq \max(1,t)$.
\end{proof}

 \Cref{lem:gbound1} implies that non-decreasing balancing functions can grow at most linearly, which is in fact enough to deduce that the associated locally-balanced Markov jump process is both $\pi$-a.e. well-defined (\Cref{prop:well-posed}) and non-explosive (\Cref{non.explo.theorem:1}).

\begin{proposition}[Well-posedness]\label{prop:well-posed}
If $g$ satisfies \Cref{ass:g_non_decr} then the following hold:
\begin{enumerate}
    \item[(a)]  $\lambda(x)<\infty$ for $\pi$-almost every $x$.
    \item[(b)] $Z_\lambda:=\int_E \lambda(x)\pi(dx)<\infty$, meaning $\tilde{\pi}(dx)=Z_{\lambda}^{-1} \lambda(x)\pi(dx)$ is a proper probability measure on $E$.
    \item[(c)] $\Gamma_g$ is a $\tilde{\pi}$-reversible discrete-time Markov transition kernel.
\end{enumerate}
\end{proposition}

\begin{proof}[Proof of \Cref{prop:well-posed}]
We first show (b), which implies part (a). 
Using \Cref{lem:gbound1}(i) and the definition of $t(x,y)$ in \eqref{eq:radon_nikodym} gives
\begin{align}
Z_\lambda=\int_{E}\lambda(x)\pi(dx)
&=
\int_{E\times E}g\left( t(x,y)\right)\gamma(x,dy)\pi(dx)\\
&\leq
\int_{E\times E}(1+t(x,y))\gamma(x,dy)\pi(dx) \nonumber \\ 
&=1+
\int_{E\times E}\pi(dy)\gamma(y,dx)
=
1+1=2<\infty \nonumber \,, 
\end{align}
as desired.
Part (c) follows from \eqref{eq:symm_measure} after noting that
\begin{align*}
[\tilde{\pi}\otimes\Gamma_g](A\times B)
&=
%\int_{A\times B}\tilde{\pi}(dx)\Gamma_g(x,dy)=
Z_{\lambda}^{-1} \int_{A\times B}\pi(dx)
g ( t(x,y)) \gamma(x,dy)
\\
&=
Z_{\lambda}^{-1} \int_{A\times B}\pi(dy)
g ( t(y,x)) \gamma(y,dx)
\\
&=
[\tilde{\pi}\otimes\Gamma_g](B\times A)
\end{align*}
for every $A,B\in \mathcal{E}$.
\end{proof}

Next we consider whether the process can explode in finite time. Recall that a jump process is said to explode if the inter-arrival times $\tau_1, \tau_2, ...$ satisfy $\sum_{k=1}^{+\infty}\tau_k<\infty$, meaning the process reaches the graveyard state $\partial$. Conversely, the process is
called non-explosive if $\sum_{k=1}^{+\infty}\tau_k=\infty$ almost surely.  To ensure non-explosivity of the process, we make the following assumption on the kernel $\Gamma_g$, which is connected with the concept of $\phi-$irreducibility (see e.g.\ \cite{meyn.tweedie:93:2}).
While irreducibility is not in general necessary for non-explosivity, it allows for a simpler treatment and it is a natural requirement in sampling applications. 

\begin{assumption}\label{phi.irr.assumption:00}
    For all $x \in \textup{supp}(\pi)$ and any $B \in \mathcal{E}$ with $\phi(B)>0$, there exists $n \in \mathbb{N}$ such that 
    $$
        \mathbb{P}_x\left( X_n \in B \right) > 0.
    $$
\end{assumption}

\begin{remark}
    For most applications in sampling (such as MCMC),  \Cref{phi.irr.assumption:00} is rather minimal. The assumption will be satisfied, for example, in the case where for all $x \in E$, $\gamma(x,\cdot)$ is supported on the entire state space, with continuous $g$, $t(x,\cdot)$ that satisfy $g(s)>0$, for all $s> 0$, and $t(x,y)>0$ for all $x,y \in E$.
\end{remark}

The next theorem shows that irreducibility of $\Gamma_g$ is sufficient to ensure that $Y$ is non-explosive for $\pi$-almost every starting state.
%We note that in this context irreducibility is a sufficient but not necessary condition to ensure non-explosivity. (I took this out as it's implied by the first sentence)
\begin{theorem}[Non-explosivity, reversibility \& ergodicity]\label{non.explo.theorem:1}
Under Assumptions \ref{lambda.positivity:00}-\ref{phi.irr.assumption:00}, the following hold:
\begin{enumerate}
    \item[(a)] $Y$ is almost surely (a.s.) non-explosive for $\pi$-a.e.\ starting state $x \in E$.
    \item[(b)] $Y$ is $\pi$-reversible, hence $\pi$-invariant.
    \item[(c)] $Y$ is ergodic for $\pi$-a.e.\ starting state $x \in E$, meaning for every $f:E\to \mathbb{R}$, with $f \in L^1(\pi)$ we have a.s.\ that
    \begin{align}\label{eq:erg_Y}
        \lim_{t\to\infty}\frac{1}{t}\int_{0}^tf(Y_s) ds
        =\int_E f(x)\pi(dx).
    \end{align}
\end{enumerate}
\end{theorem}

\begin{proof}[Proof of  \Cref{non.explo.theorem:1}]
For (a), by definition $(Y_t)_{t\geq 0}$ explodes if and only if $\sup\{n\,:\,T_n \leq t\}=\infty$ for some $t<\infty$, i.e.\ if $\sum_{k=1}^\infty \tau_k<\infty$. 
Conditionally on $X$, by Theorem 2.3.2 of \cite{norris:97}, a.s. $\sum_{k=1}^\infty \tau_k<\infty$ if and only if $\sum_{k=1}^{+\infty}\lambda(X_k)^{-1}<\infty$, thus unconditionally
   \begin{equation*}
       \mathbb{P}_x\left(\sum_{k=1}^\infty \tau_k<\infty\right)
       =
       \mathbb{P}_x\left(\sum_{k=1}^{+\infty}\frac{1}{\lambda(X_k)}<\infty\right).
   \end{equation*}
By \Cref{prop:well-posed}(c) $\Gamma_g$ is $\tilde{\pi}$-invariant. Since $\Gamma_g$ is also $\phi$-irreducible for some non-trivial $\phi$, by Theorem 1 of \cite{asmussen2011new}, for $\tilde{\pi}$-almost any $x \in E$ and any $\tilde{\pi}$-integrable $f:E \to [0,\infty)$
\begin{equation*}
    \frac{1}{N}\sum_{k=1}^{N}f(X_k)\xrightarrow{N \rightarrow \infty}\int f(x)\tilde{\pi}(dx)
\end{equation*} 
almost surely. Setting $f(x) := 1/\lambda(x)$ gives   
\begin{equation*}
    \frac{1}{N}\sum_{k=1}^{N}\frac{1}{\lambda(X_k)}\xrightarrow{N \rightarrow \infty}\int\frac{1}{\lambda(x)}\frac{1}{Z_{\lambda}}\lambda(x)\pi(dx)=\frac{1}{Z_{\lambda}}>0
\end{equation*} 
almost surely, implying that $\mathbb{P}_x(\sum_{k=1}^{+\infty}\lambda(X_k)^{-1}<+\infty)=0$ for $\tilde{\pi}$-almost any $x$. 

For (b) it is classical that if $Y$ is non-explosive then the unique solution to the Kolmogorov backward equation is stochastic, see e.g.\ Section X.3 of \cite{feller1991introduction}. Combining with \Cref{prop:reversible_jumpkernel} implies the result, a proof of which can be found as e.g.\ Proposition 2 of \cite{serfozo2005reversible}.

% Under (b), then for almost all events $\omega$ there exists a compact set $K=K(\omega)$ and such that $X_k \in K$ for infinitely many $k$. Since $\lambda$ is bounded on $K$, there exists an $M>0$ such  that 
% \begin{equation*}
%     \frac{1}{\lambda(X_k)} \geq \frac{1}{M}
% \end{equation*}
% for any $k$ such that $X_k \in K$. Therefore $\sum_{k=1}^{+\infty}\lambda(X_k)^{-1}=+\infty$.

Consider now part (c). 
Without loss of generality, we will assume that $f(x) \geq 0$ for all $x \in E$, and the result will follow by considering the positive and negative part of a general $f \in L^1(\pi)$. By \Cref{def:LBMJP}, we have
    \begin{align*}
        \frac{1}{t}\int_{0}^tf(Y_s) ds
        =
        \frac{1}{t}
        \sum_{i=1}^{n(t)}\tau_i f(X_{i-1}) + \frac{1}{t}
        (t-T_{n(t)})f(X_{n(t)}),    
       % =
        %\lim_{n\to\infty}\frac{\sum_{i=1}^n\tau_i f(X_{i-1})}{\sum_{i=1}^n\tau_i}
    \end{align*}
 where $n(t):=\sup\{n\,:\,T_n \leq t\}$. Using the fact that a.s.\ $t-T_{n(t)} \leq \tau_{n(t)+1}$, and that $T_{n(t)} \leq t \leq T_{n(t)+1}$, we write

 \begin{equation}\label{eq:erg_Y:0}
    \frac{1}{T_{n(t)+1}}
        \sum_{i=1}^{n(t)}\tau_i f(X_{i-1}) \leq \frac{1}{t}\int_{0}^tf(Y_s) ds \leq    \frac{1}{T_{n(t)}} \sum_{i=1}^{n(t)+1}\tau_i f(X_{i-1}).
 \end{equation}
 
By the proof of Theorem 1 of \cite{asmussen2011new}, $(X_n)_{n\in\mathbb{N}}$ is $\tilde{\pi}$-irreducible, i.e. one can choose $\tilde{\pi}$ to be the non-trivial measure $\phi$ in the definition of $\phi$-irreducibility. Then, $((X_{n-1},\tau_n))_{n\in\mathbb{N}}$ is a Markov chain that is also $\phi$-irreducible with measure $\phi(dx, d\tau) := \tilde{\pi}(dx) d\tau$ on $E \times [0,+\infty)$, with $d\tau$ denoting the Lebesgue measure. It also has an invariant distribution $\hat{\pi}$ defined on $E \times [0,+\infty)$, with $X$-marginal equal to $\tilde{\pi}$ and conditional distribution of $\tau|X$ equal to $\textup{Exp}\left( \lambda(X) \right)$, meaning
\begin{equation*}
\hat{\pi}(dx,d\tau) = \frac{1}{Z_{\lambda}}  \lambda^2(x) \exp\left\{ -\lambda(x) \tau \right\} d \tau \ \pi(dx).  
\end{equation*}
By Theorem 1 of \cite{asmussen2011new}, using the Law of Large Numbers for the functions $g(x,s)=f(x) s$, and $h(x,s)=s$ we deduce that for $\tilde{\pi}$-a.e. starting points $x \in E$, a.s.
\begin{align}\label{eq:erg_Y:1}
   \lim_{n \rightarrow \infty} \frac{1}{n}\sum_{i=1}^{n}\tau_i f(X_{i-1}) 
   = 
   \mathbb{E}_{\hat{\pi}}[g]
   &= \int_E \int_0^{+\infty} f(x) s \frac{1}{Z_{\lambda}} \lambda^2(x) \exp\left\{ -\lambda(x) s \right\} ds \ \pi(dx) \nonumber
   \\
   &= \frac{1}{Z_{\lambda}} \int_Ef(x)\pi(dx), 
\end{align}
and
\begin{equation}\label{eq:erg_Y:2}
    \lim_{n \rightarrow \infty} \frac{1}{n}T_n = \lim_{n \rightarrow \infty} \frac{1}{n}\sum_{i=1}^{n}\tau_i = \mathbb{E}_{\hat{\pi}}\left[ h \right]= \int_E \int_0^{+\infty} s \frac{1}{Z_{\lambda}} \lambda^2(x) \exp\{ -\lambda(x)s \}ds \pi(dx) = \frac{1}{Z_{\lambda}}.
\end{equation}

Finally, observe that since $\lambda(x)>0$ for all $x \in E$, we have a.s. $n(t) \xrightarrow{t \rightarrow + \infty} +\infty$. Combining \eqref{eq:erg_Y:1} and \eqref{eq:erg_Y:2}, we see that on letting $t \rightarrow +\infty$ on the left-hand side of \eqref{eq:erg_Y:0} we get
\begin{align*}
    \limsup_{t \rightarrow +\infty} \frac{1}{T_{n(t)+1}}
    \sum_{i=1}^{n(t)}\tau_i f(X_{i-1}) 
    &= 
    \limsup_{t \rightarrow + \infty} \frac{n(t)}{n(t)+1}\frac{1}{\left({n(t)+1}\right)^{-1}T_{n(t)+1}} n(t)^{-1} \sum_{i=1}^{n(t)}\tau_i f(X_{i-1}) 
    \\
    &= \int_Ef(x)\pi(dx)
\end{align*}
and a similar argument shows that for the right-hand side of  \eqref{eq:erg_Y:0} we get
\begin{equation*}
    \liminf_{t \rightarrow +\infty} \frac{1}{T_{n(t)}} \sum_{i=1}^{n(t)+1}\tau_i f(X_{i-1})= \int_Ef(x)\pi(dx).
\end{equation*}
The result follows on noting that since $0<\lambda(x)< \infty$ for $\pi$-a.e.\ $x$, any set of full $\tilde{\pi}$-measure is also of full $\pi$-measure, therefore \eqref{eq:erg_Y} holds for $\pi$-a.e.\ starting point.
\end{proof}

Next, we consider the average number of jumps of the LBMJP in a finite time horizon. The assumption that this is finite is required for many theoretical results of interest.  With a view towards using the theory developed in \cite{davis:84} for piecewise-deterministic Markov processes, we present the following.

\begin{lemma}\label{lem:finite.expected.jumps:0}
Let $n_T$ be the number of jumps of a LBMJP until time $T\in(0,\infty)$.  Under Assumptions \ref{lambda.positivity:00}-\ref{phi.irr.assumption:00}, there exists $ S \in \mathcal{E}$ with $\pi(S)=1$ such that for all $T>0$ and $x \in S$,
    \begin{equation}\label{finite.expected.jumps:1}
        \mathbb{E}_x\left[ n_T \right] < \infty.
    \end{equation}
\end{lemma}

\begin{proof}[Proof of \Cref{lem:finite.expected.jumps:0}]
    The jumps follow a non-homogeneous Poisson process with rate $\lambda(Y_t)$. Standard results (see e.g.\ \cite{bierkens.roberts.zitt:19}) show that
    \begin{equation*}
        \mathbb{E}_x\left[ n_T \right] = \mathbb{E}_x\left[ \int_0^T \lambda(Y_t)dt \right].
    \end{equation*}
    Assume that the starting point $x$ of the process is distributed according to $\pi$, which is invariant for the process $Y$ from \Cref{non.explo.theorem:1}(b). Then
    \begin{equation*}
        \mathbb{E}_{\pi}\left[ n_T \right] = \mathbb{E}_\pi \left[ \int_0^T \lambda(Y_t)dt  \right] = \int_0^T \mathbb{E}_{\pi}\left[ \lambda(Y_t) \right] dt = \int_0^T \int_E \lambda(x) \pi(dx) dt = T \int_E \lambda(x) \pi(dx) < \infty
    \end{equation*}
since $\int_E \lambda(x) \pi(dx) < \infty$ from \Cref{prop:well-posed}. The result follows.
\end{proof}

To conclude with the well-posedness of the LBMJP, we address the concern that $Y$ might hit a state $x$ with $\lambda(x)=+\infty$ when starting from some zero measure set. To avoid this, we make the following assumption.

\begin{assumption}\label{ass.finite.lambda}
    For all $x \in \textup{supp}(\pi)$, and any $A \in \mathcal{E}$ with $\pi(A)=0$, we have $\Gamma_g(x,A)=0$. 
\end{assumption}

\begin{remark}
    Under Assumption \ref{lambda.positivity:00}, we have that $\pi(A)=0 \iff \tilde{\pi}(A)=0$. Therefore, Assumption \ref{ass.finite.lambda} asks that for any $x \in \textup{supp}(\tilde{\pi})$, $\Gamma_g(x,\cdot)$ is absolutely continuous with respect to $\tilde{\pi}$. Since $\tilde{\pi}$ is also the invariant measure of the kernel $\Gamma_g$, this will certainly hold for $\tilde{\pi}$-almost all states $x$, and therefore for $\pi$-almost all states $x$. Therefore, Assumption \ref{lambda.positivity:00} merely imposes an extra condition on a zero $\pi$-measure set.
\end{remark}

Under \Cref{ass.finite.lambda}, we have the following strengthening of Theorem \ref{non.explo.theorem:1} parts (a) and (c) from holding for $\pi$-a.e. starting state to every starting state $x$ in the support of $\pi$, along with some additional stability of the process.

\begin{theorem}[Stability and ergodicity from any starting state]\label{stability.theorem.final:00}
    Under Assumptions  \ref{lambda.positivity:00}-\ref{ass.finite.lambda}, from any starting state $x \in \textup{supp}(\pi)$ the following hold:
    \begin{enumerate}
        \item[(a)] The process $Y_t$ is a.s.\ non-explosive. %from any starting state $x \in \textup{supp}(\pi)$.
    \item[(b)] $Y_t \in \textup{supp}(\pi)$ and $\lambda(Y_t) < \infty$. %from any starting state $x \in \textup{supp}(\pi)$ and all $ t \geq 0$. 
    \item[(c)] For all $f \in L^1(\pi)$, a.s.\
    \begin{equation}\label{eq:erg_Y.again}
      \lim_{t \rightarrow \infty}\frac{1}{t} \int_0^t f(Y_s)ds = \int_E f(x)\pi(dx).  
    \end{equation}
\end{enumerate}
\end{theorem}

\begin{proof}[Proof of \Cref{stability.theorem.final:00}]
Let $\Lambda_{<\infty} = \{ y \in E : \lambda(y) < \infty \}$. By \Cref{prop:well-posed}(a), $\pi(\Lambda_{<\infty} \cap \textup{supp}(\pi)) = 1$. Therefore, by Assumption \ref{ass.finite.lambda}, and since we start from $x \in \textup{supp}(\pi)$, almost surely the process $X_1, X_2, \dots$ will stay on $\Lambda_{<\infty} \cap \textup{supp}(\pi)$ for all $n \in \mathbb{N}$, and, therefore, the same will hold for the process $Y_t$ for all $t \geq 0$.
Using the same argument as in \Cref{non.explo.theorem:1}(b), we get that the process is $\pi$-reversible and invariant. Finally, \eqref{eq:erg_Y.again} follows by the same arguments of the proof of \Cref{non.explo.theorem:1}(c) by appealing to Theorem 2 of \cite{asmussen2011new}.
\end{proof}

\begin{remark}
    As $\lambda \in L^1(\pi)$ (by  \Cref{prop:well-posed}), and $\lambda(Y_t)<\infty$ a.s.\ for all $t \geq 0$ (by \Cref{stability.theorem.final:00}), we can always consider a version of $\lambda \in L^1(\pi)$ such that $\lambda(x)<\infty$ for all $x \in E$, and this will not affect the behaviour of the process. %From this point forward we therefore assume without loss of generality that $\lambda(x)< \infty$ for all $x \in E$.
\end{remark}

%As a result of \Cref{stability.theorem.final:00}, from this point forward we will assume that the process is non-explosive, and that the starting point $x$ lies in the support of $\pi$.

\subsection{Feller property}\label{sec::Feller}

For a continuous time Markov process with transition kernel $P_t(x,\cdot)$ at time $t$, let the corresponding semigroup $(P_t)_{t\geq 0}$ be defined through the relation $P_tf(x) := \int f(y)P_t(x,dy)$.  We will refer to a Markov process for which $P_tf(x) \to f(x)$ as $t \to 0$ point-wise for all $x \in E$ as \emph{strong Feller} if when $f:E\to \mathbb{R}$ is bounded then $P_tf \in C_b(E)$. If $f \in C_b(E) \implies P_t f \in C_b(E)$ then we will call the process weak Feller. 
In general a Markov jump process will not be strong Feller, since at time $t$ if there is some non-zero probability that no jumps have occurred then $P_t f$ is a convex combination of functions including $f$, which may not be continuous.  %We show below that when $g$ is bounded then a LBMJP is Feller--Dynkin.  In fact when it is unbounded this need not be the case, which relates to the fact that in the case of unbounded $g$ the process can move a large distance in a short space of time.  We explore this phenomenon further in  \Cref{subsec:unif_erg}, showing examples in which the process is uniformly ergodic even when $E$ is unbounded.  
%Below we provide simple sufficient conditions under which a LBMJP is weak Feller. %, which allow for unbounded $g$. 

%\begin{proposition}[Feller--Dynkin]\label{bounded.rate.feller:1}
%   Let $\left( Y_t \right)_{t \geq 0}$ be a LBMJP. If $g$ is bounded above then the process is Feller--Dynkin.
%\end{proposition}
%\begin{proof}[Proof of \Cref{bounded.rate.feller:1}]
%    Set $\bar{\lambda}:=\sup_t g(t)$, note that $\lambda(x) \leq \bar{\lambda}$, and let $f \in C_0(\mathbb{R}^d)$. If the process starts from $x \in E$, let $T_1 \sim \exp(\lambda(x))$ be the first jump time of the process. We have that for any
%$t \geq 0$ and $x \in \mathbb{R}^d$
%\begin{align*}
%\left| P_tf(x)-f(x) \right| 
%&= 
%\left| \mathbb{E}_x\left[ f(Y_t) -f(x) \right] \right| \leq \mathbb{E}_x\left[ \left| f(Y_t) -f(x) \right| \right]  
%\\
%&=\mathbb{P}\left( T_1\leq t \right)  \mathbb{E}_x\left[ \left| f(Y_t) -f(x) \right| | T_1 \leq t \right] + \mathbb{P}\left( T_1 > t \right)  \mathbb{E}_x\left[ \left| f(Y_t) -f(x) \right| | T_1 > t \right] 
%\\
%&= 
%\mathbb{P}\left( T_1\leq t \right)  \mathbb{E}_x\left[ \left| f(Y_t) -f(x) \right| | T_1 \leq t \right] \leq \left( 1- \exp\{ \bar{\lambda} t \} \right) 2 \| f \|_{\infty},
%\end{align*}
%    meaning that
%    \begin{equation*}
%        \| P_tf-f \|_{\infty} \leq  \left( 1- \exp\{ \bar{\lambda} t \} \right) 2 \| f \|_{\infty} \xrightarrow{t \rightarrow 0} 0,
%    \end{equation*}
%    which implies that $P_t f \in C_0(E)$ as required.
%\end{proof}

To establish the weak Feller property in some generality, we make the following assumptions.

\begin{assumption}\label{regularity.ass:1}
The following hold.
\begin{enumerate}
   \item Let $(U,\mathcal{B}(U),\nu)$ be a probability space, and $h:E \times U \rightarrow E$ such that for any $x \in E$, $\gamma(x,\cdot)=\mathbb{P}_{u \sim \nu}\left(  h(x,u) \in \cdot \right)$. Assume that for all $u \in U$ the function $x \rightarrow h(x,u)$ is continuous.

%\begin{assumption}\label{pi.zero.order.assumptions:1}
\item The Radon--Nikodym derivative $t:E \times E \rightarrow \mathbb{R}_{\geq 0}$ as in (\ref{eq:radon_nikodym}) is continuous. 
%\end{assumption}

%\begin{assumption}\label{continious.g.assumption:1}
%  \item   The balancing function $g$ is continuous.
  \end{enumerate}
\end{assumption}

Assumption \ref{regularity.ass:1}  
will be satisfied in many applied settings. Assumption \ref{regularity.ass:1}.1, for example, will be verified for a kernel $\gamma$ that induces a random walk. In particular, using the notation of Assumption \ref{regularity.ass:1}.1, we then have $U=E$, $h(x,u)=x+u$, so $h(\cdot,u)$ is continuous.  Assumption \ref{regularity.ass:1}.2 will be satisfied if for example $\pi$ and $\gamma$ admit continuous densities with respect to a common reference measure.

\begin{proposition}[Weak Feller]\label{Feller.proposition:1}
Under Assumptions \ref{ass.finite.lambda} and \ref{regularity.ass:1}, a locally-balanced Markov jump process is weak Feller.
\end{proposition}

The proof of \Cref{Feller.proposition:1} relies on the following lemma.

\begin{lemma}\label{lemma.continuous.rate:0}
   Under Assumptions \ref{ass.finite.lambda} and \ref{regularity.ass:1}, the rate $\lambda$ of the LBMJP is a continuous function. 
\end{lemma}

\begin{proof}[Proof of \Cref{lemma.continuous.rate:0}]
     Using Assumption \ref{regularity.ass:1} we write for any $y \in E$,
    \begin{equation*}
        \lambda(y)=\int_E g \circ t(y,z) \gamma(y,dz)=\int_{U}g \circ t(y,h(y,u)) \nu(du).
    \end{equation*}
Let us fix $x \in E$ and $u \in U$. From Assumption   \ref{regularity.ass:1}.2, we have $\lim_{y \rightarrow x} g \circ t(y,h(y,u))=g \circ t(x,h(x,u))$. For some $\epsilon > 0$, consider the ball $B(x,\epsilon)$ with radius $\epsilon$, centred at $x$.  Then, using Assumption \ref{regularity.ass:1}.2, there exists an $M_x$ such that $t(y,h(y,u)) \leq M_x$ for all $y \in B(x,\epsilon)$. By \Cref{lem:gbound1} this implies that for all $y \in B(x,\epsilon)$
\begin{align*}
    g \circ t(y,h(y,u))\leq 1+ t(y,h(y,u)) \leq 1+M_x.
\end{align*}
By the Bounded Convergence Theorem this then implies that $\lim_{y \rightarrow x} \lambda(y) =\lambda(x)$.
\end{proof}

\begin{proof}[Proof of  \Cref{Feller.proposition:1}]
We construct the process using a sequence of i.i.d. variables $E_1, E_2, \\
\dots \sim \text{Exp}(1)$ and $U_1,U_2,... \sim \nu$ % as in (\ref{regularity.ass:1} 
 in the following way. Assume that the process starts from $x \in E$. We set $\tau^x_1=\lambda(x) E_1$ and $\tau^x_1=T^x_1$ the first jump time. $Y_s=x$ for $s \in [0,\tau^x_1)$. Then the process jumps to $h(x,U_1)$ so $Y^x_{T_1}=h(x,U_1)$. The process is then constructed inductively, i.e. if it has been constructed until the $n$'th jump time $T^x_n$ then we set $\tau^x_{n+1}=E_{n+1}\lambda(Y^x_{\tau_n})$, $T^x_{n+1}=T^x_n+\tau^x_{n+1}$, $Y^x_{s}=Y^x_{T^x_{n}}$ for $s \in [T^x_{n},T^x_{n+1})$ and $Y^x_{T^x_{n+1}}=h(Y^x_{T^x_n},U_{n+1})$.

Let us fix $t \geq 0$, $x \in E$ and let us fix a specific configuration of these $E_1, E_2,...$ and $U_1,U_2,...$. Our goal will be to prove that a.s. on the configuration of $E$'s and $U$'s, $\lim_{y \rightarrow x}Y^y_t=Y^x_t$. If we have that then for any $f \in C_b(E)$, we have $\lim_{y \rightarrow x}f(Y^y_t)=f(Y^x_t)$ a.s. and from bounded convergence, $P_tf(y)=\mathbb{E}_y\left[ f(Y_t) \right] \xrightarrow{y \rightarrow x}\mathbb{E}_x\left[ f(Y_t) \right]=P_tf(x)$ which proves that $P_tf$ is continuous.

We now turn on proving that a.s. $Y^y_t\xrightarrow{y \rightarrow x}Y^x_t$. Due to non-explosivity of the process, a.s. there exists $n \in \mathbb{N}$ such that $t \in [T^x_n,T^x_{n+1})$ and since a.s. $T^x_n \neq t$, $t \in (T^x_n,T^x_{n+1})$. Consider the discrete time chain $\left(X^y_k\right)_{k \geq 0}$
defined as $X^y_k=Y^y_{\tau^y_k}$. Since $h(\cdot,u)$ is continuous for all $u$, we observe that 
\begin{equation*}
  \lim_{y \rightarrow x}X^y_1=h(y,U_1)=h(x,U_1)=X^x_1,
\end{equation*}
and since
\begin{equation*}
X^y_{k+1}=h(X^y_k,U_{k+1}).
\end{equation*}
using induction we get that for any $k \in \mathbb{N}$, $X^y_k$ is continuous with resepct to $y$. 
Furthermore, we can write 
\begin{equation*}
    T^y_1=\lambda(y)E_1
\end{equation*}
which is continuous with respect to $y$ since $\lambda$ is continuous. We also observe that
for any $k \in \mathbb{N}$, 
\begin{equation*}
    T^y_{k+1}=T^y_k+\lambda(X^y_k)E_{k+1}
\end{equation*}
so if $T^y_k$ is continuous, then $T^y_{k+1}$ is as well. Via induction, $T^y_k$ is continuous with respect to $y$ for all $k$. Therefore, since $T^y_n \in \mathbb{N}$ for all $y$, there exists $\delta >0$ such that if $y \in B(x,\delta)$, then $T^y_n=T^x_n$. This further implies that $t \in (T_n^y,T_{n+1}^y)$ for any $y \in B(x,\delta)$. Therefore,
\begin{equation*}
    Y^x_t=X^x_{Y^x_n}=\lim_{y \rightarrow x} X^y_{T^x_n}=\lim_{y \rightarrow x} X^y_{T^y_n}=\lim_{y \rightarrow x} Y^y_t.
\end{equation*}
This proves that $P_tf$ is continuous. Finally, since $f$ is bounded and $(P_t)_{t\geq 0}$ defines a contraction semigroup then $P_tf(x)=\mathbb{E}_x\left[ f(Y_t) \right]$ is also bounded, meaning that $P_tf \in C_b(E)$. 

It remains to prove that for any bounded $f: E \rightarrow \mathbb{R}$, and any $x \in E$ we have
\begin{equation*}
    P_tf(x) \xrightarrow{t \rightarrow 0} f(x).
\end{equation*}
To prove this, we observe that
\begin{align*}
    \left|P_tf(x)-f(x)\right|&=\left|\mathbb{E}_x\left[ f(Y_t) -f(x) \right]\right| \\
    &\leq \mathbb{E}_x\left[ \left| f(Y_t) -f(x) \right| \right] \\
    &=\mathbb{E}_x\left[ \left| f(Y_t) -f(x) \right| \mathbb{I}_{\tau_1 \leq h} \right]+ \mathbb{E}_x\left[ \left| f(Y_t) -f(x) \right| \mathbb{I}_{\tau_1 > h} \right] \\
    &=\mathbb{E}_x\left[ \left| f(Y_t) -f(x) \right| \mathbb{I}_{\tau_1 \leq h} \right]\\
    &\leq 2\| f \|_{\infty} \mathbb{P}_x\left( \tau_1 \leq h \right)\\
    &= 2\| f \|_{\infty} \left(1-\exp\left\{ -\lambda(x)t \right\} \right)\xrightarrow{t \rightarrow0}0.
\end{align*}
This completes the proof.
\end{proof}

\begin{comment}
\subsection{Stability properties}

\begin{definition}
    A Markov process $\left( Y_t \right)_{t \geq 0}$ is called aperiodic if there exists a petite set $C$ and $T>0$ such that for all $x \in C$ and $t \geq 0$, $\mathbb{P}_x\left( Y_t \in C \right) >0$.
\end{definition}
\begin{assumption}\label{ass.upper.bound.rate.petite:0}
    There exists a petite set $C$ and $\lambda^{\ast} < \infty$ such that for any $x \in C$, $\lambda(x) \leq \lambda^{\ast}$.
\end{assumption}

We then have the following result. We state it here for the LBMJP, but observe that the same would hold for any Markov Jump Process.

\begin{proposition}\label{aperiodicity:1}
    If Assumption \ref{ass.upper.bound.rate.petite:0} holds, then the LBMJP is aperiodic.
\end{proposition}

%The proof follows an argument found in the proof of Theorem 5 of \cite{bierkens.roberts.zitt:19}.
\begin{proof}[Proof of Proposition \ref{aperiodicity:1}]
    For any $x \in C$ and any $s \in [0,1]$, by considering the event where no jump occurred before time $1$, we get
    \begin{equation*}
        \mathbb{P}_x\left( Y_s \in C \right) \geq \exp\left\{  -\lambda(x) \right\} \geq \exp\left\{ -\lambda^{\ast} \right\}.
    \end{equation*}
    Let $t \geq 1$. Set $s=\frac{t}{\lceil t \rceil} \in (0,1)$. Then 
    \begin{equation*}
        \mathbb{P}_x\left( Y_t \in C \right) \geq \left(\exp\{ -\lambda^{\ast} \}\right)^{\lceil t \rceil} >0.
    \end{equation*}
\end{proof}
\end{comment}

\subsection{Generator}\label{sec::Generator}
In this sub-section we study the generator of a locally-balanced Markov jump process. This will become important in the next sections when we discuss ergodicity properties and the behaviour of the process in the limit when the size of jumps becomes small.
Write $L$ to denote the operator such that for any  bounded function $f:E \rightarrow \mathbb{R}$, 
we have
\begin{equation}\label{weak.generator.first}
    Lf(x):=\int_E g \circ t(x,y) \left( f(y) - f(x) \right) \gamma(x,dy) = \lambda(x) \left( \Gamma_gf(x) - f(x) \right),
\end{equation}
where $\Gamma_g$ is as in \eqref{def.gamma.g:1}.
It is well known (see e.g.\ Theorem 3.1 \cite{ethier2009markov}) that under suitable conditions the operator defined by
\begin{equation*}
    Af(x)=a(x) \left( \Gamma f(x) -  f(x) \right)
    \end{equation*}
is the (strong) generator of a Markov Jump Process with jump kernel of the form $a \cdot \Gamma$ (see for example \cite{ethier2009markov}). 
Some of these conditions are, however, quite restrictive for our setting. But taking advantage of results of \cite{davis:84} for piecewise deterministic Markov processes, we have the following. 

\begin{theorem}[Weak generator]\label{weak.generator.domain:0}
    Let Assumptions \ref{lambda.positivity:00}-\ref{ass.finite.lambda} hold. Then for any bounded $f:E \rightarrow \mathbb{R}$, $t \geq 0$ and $x \in \textup{supp}(\pi) \cap S$, with $S$ as in Lemma \ref{lem:finite.expected.jumps:0}, 
    the process
    \begin{equation}\label{martingale:001}
        M_t=f(Y_t)-f(x)-\int_0^tLf(Y_s)ds , \ \ t \geq 0
    \end{equation}
    is a martingale with respect the natural filtration of $Y$. Therefore, the domain $\mathcal{D}(L)$ of the weak generator of $Y$ contains all bounded functions $B(E)$, and the {\it weak generator} of the LBMJP restricted on  $B(E)$ is the operator $L$.
\end{theorem}

\begin{proof}[Proof of \Cref{weak.generator.domain:0}]
    From \Cref{lem:finite.expected.jumps:0}, $\mathbb{E}_x\left[ n_T \right]<\infty$ for all $T \geq 0$. Furthermore, since $f$ is bounded, for any $T>0$, letting $Y_{t^-}:=\lim_{s \rightarrow t^-}Y_s$, we have 
    \begin{equation*}
        \mathbb{E}_x \left[ \sum_{k=1}^{n_T} \left|f(Y_{T_k}) - f(Y_{T_k^-})\right| \right] \leq 2 \| f \|_{\infty} \mathbb{E}_x\left[ n_T \right] < \infty.
    \end{equation*}
    From Theorem 5.5 of \cite{davis:84} we get that $M_t$ is a local Martingale. We further observe that for any $y \in E$,
    \begin{equation*}
|Lf(y)|=\left|\int_E\left( f(y)-f(x) \right) g \circ t(x,y) \gamma(x,dy)\right| \leq 2 \| f \|_{\infty} \lambda(x) 
    \end{equation*}
    therefore for any $T \geq 0$
    \begin{equation*}
        \sup_{t \leq T}\left|\int_0^t Lf(Y_s) ds\right| \leq 2 \| f \|_{\infty} \int_0^T \lambda(Y_s)ds
    \end{equation*}
    which has finite expectation since $x \in \mathcal{N}$ and from \Cref{lem:finite.expected.jumps:0}. Since $f$ is bounded, we get that $M$ is a true martingale.
\end{proof}

We also have the following point-wise limit, which will prove useful in Section \ref{sec::Mixing}.

\begin{proposition}\label{point.wise.limit.generator:1}
 Assume that Assumption \ref{ass.finite.lambda} holds. For any bounded $f:E \rightarrow \mathbb{R}$ and any $x \in E$
\begin{equation*}
    \lim_{h \rightarrow 0}\frac{\mathbb{E}_x\left[ f(Y_h)-f(x)  \right]}{h}=Lf(x).
\end{equation*}
\end{proposition}

\begin{proof}[Proof of \Cref{point.wise.limit.generator:1}]
The proof is presented in \ref{appendix:pointwise.limit:000}.
\end{proof}

Next we focus on the case of bounded $g$.  Here the generator has numerous additional regularity properties.

\begin{proposition}\label{prop:bounded.generator:1}
Under Assumption \ref{lambda.positivity:00}, if $g$ is bounded above then the following hold:
\begin{enumerate}
    \item[(a)] $L$ is a bounded operator on $L^2(\pi)$.
    \item[(b)] The domain $\mathcal{D}(L) = L^2(\pi)$.
    \item[(c)] $L$ generates a uniformly continuous contraction semi-group $(P_t)_{t\geq 0}$. 
    \item[(d)] The semi-group can be written
    $$
    P_t := \exp(tL) = \sum_{n=0}^\infty \frac{t^n}{n}L^n.
    $$
\end{enumerate}
\end{proposition}

\begin{proof}[Proof of \Cref{prop:bounded.generator:1}]
For (a) set $\lambda:= \sup_x \lambda(x)$. Then we can equivalently define the process as evolving with constant rate (see Ethier \& Kurtz Section 4.2) and write the generator as
$
Lf(x) = \lambda \int [f(y)-f(x)] \bar{\Gamma}_g(x,dy),
$
where
$$
\bar{\Gamma}_g(x,dy) := \left( 1 - \frac{\lambda(x)}{\lambda}\right)\delta_x(dy) + \frac{\lambda(x)}{\lambda}\Gamma_g(x,dy).
$$
Note that 
\begin{comment}
$$
\begin{aligned}
\lambda(x)\pi(dx)\Gamma_g(x,dy) 
&= 
g\left(\frac{\pi(dy)\gamma(y,dx)}{\pi(dx)\pi(dx)\gamma(x,dy)}\right)\gamma(x,dy) \\
&=
g\left(\frac{\pi(dx)\gamma(x,dy)}{\pi(dy)\gamma(y,dx)}\right)\pi(dy)\gamma(y,dx) \\
&=
\lambda(y)\pi(dy)\Gamma_g(y,dx),
\end{aligned}
$$
\end{comment}
$\bar{\Gamma}_g$ is $\pi$-reversible and 
$
\bar{\Gamma}_g(x,E) = 1,
$
implying that the corresponding operator $\bar{\Gamma}_gf(x) := \int f(y)\bar{\Gamma}(x,dy)$ for $f \in L^2(\pi)$ satisfies $\|\bar{\Gamma}_g\| \leq 1$.  For any $f \in \mathcal{D}(L)$ it therefore holds that
$$
\begin{aligned}
\|Lf\|^2 &= \int (Lf(x))^2 \pi(dx)
= \int \lambda^2 (\bar{\Gamma}_g f(x) - f(x))^2 \pi(dx)
%&= \int \lambda^2 \left( \bar{\Gamma}_g f(x) \right)^2 \pi(dx)
%+ \int \lambda^2 f(x)^2 \pi(dx)
%- 2\int \lambda^2 f(x) \bar{\Gamma}_g f(x) \pi(dx) \\
\leq \lambda^2\|\bar{\Gamma}_g f\|^2 + \lambda^2\|f\|^2 + 2\lambda^2\|f\|\|\bar{\Gamma}_g f\|,
%&\leq 4\lambda^2\|f\|^2,
\end{aligned}
$$
implying that $\|L\| <\infty$ as required.  Parts (b), (c) and (d) follow directly from (a), see e.g.\ Theorem 1.2 of \cite{pazy2012semigroups}.
\end{proof}

\section{Mixing properties}\label{sec::Mixing}

Here we study exponential and uniform ergodicity properties of a locally-balanced Markov jump process. In Section \ref{sec:expo.ergo:000} we define a weak notion of spectral gap, which implies exponential ergodicity for the process. We then show that in the case of bounded $g$ there is an equivalence between the existence of a spectral gap for locally-balanced processes and Metropolis--Hastings algorithms in Section \ref{subsubsec:lb_vs_mh}.  After this we provide tools to compare two different locally-balanced processes with different choices of $g$ in Section \ref{subsubsec:compare}.  Finally, in Section \ref{subsec:unif_erg} we show conditions under which a locally-balanced process can be uniformly ergodic.

\subsection{Spectral gaps and exponential ergodicity}\label{sec:expo.ergo:000}

%\sjlcom{Currently this section is in a sense formal, because generally the Dirichlet form is defined in terms of the strong generator. I believe it is sufficient to establish that a process os $C_b$-Feller to define a strong generator. But with what we have at the moment and the weak generator I don't know if we can work directly with Dirichlet forms. I think we can probably show $C_b$-Feller, it's weaker than Feller--Dynkin. In any case, if we were forced to assume Feller--Dynkin for this section then it wouldn't be the worst because we treat the case that is not in the uniform ergodicity section.

%UPDATE: The arguments do actually seem to work with what we have, we just should be careful to call it a modified spectral gap of sorts and stick to bounded functions.  I have updated the text to reflect this.}

For any $f \in \mathcal{B}(E)$ we define a Dirichlet form associated with $L$ (defined in \eqref{weak.generator.first}) as
\begin{equation}
    \mathcal{E}(L,f) := \langle f,-Lf\rangle_\pi.
\end{equation}

Note that typically the Dirichlet form is defined using the strong generator.  Here we instead use the weak generator and rely on the point-wise convergence established in \Cref{point.wise.limit.generator:1}, which will be sufficient for our needs whilst also allowing us to avoid questions about the domain of the strong generator.  We will define a form of spectral gap for the process as
\begin{equation}\label{our.spectral}
    \gamma_L := \inf_{f \in \mathcal{B}(E)}\frac{\mathcal{E}(L,f)}{\text{Var}_\pi(f)}.
\end{equation}
We note the small difference to the usual definition of spectral gap $\text{G}(L)$, defined as 
\begin{equation}\label{general.spectral}
    \text{G}(L) := \inf_{f \in L^2(\pi)}\frac{\mathcal{E}(L,f)}{\text{Var}_\pi(f)}
\end{equation}

when $L$ is the strong generator.  For a LBMJP the seemingly weaker condition $\gamma_L > 0$ in fact directly implies exponential ergodicity in total variation distance provided that the process is initialised from a distribution $\mu$ that is absolutely-continuous with respect to $\pi$ with finite $\chi^2-$divergence, as stated below.

\begin{proposition}[Exponential ergodicity under spectral gap] \label{prop:geometric_ergodicity}
Let Assumptions \ref{lambda.positivity:00}-\ref{ass.finite.lambda} hold. If an LBMJP with balancing function $g$ and weak generator $L$ is such that $\gamma_L > 0$ and $Y_0 \sim \mu$ with $\mu \ll \pi$, then 
\begin{equation}
    \|\mu P_t - \pi\|_{TV} \leq \frac{1}{2}\chi^2(\mu\|\pi)^{\frac{1}{2}}\exp\{-\gamma_L t\}.
\end{equation}
    
\end{proposition}

\begin{proof}
    Let $f \in B(E)$. We have $\mathcal{E}(L,f) \geq \gamma_L\text{Var}_\pi(f)$. By  \Cref{stability.theorem.final:00}, since $\pi$ is invariant for the process, it holds that
    $$
    \frac{d}{dt}\text{Var}_\pi(P_tf) = \frac{d}{dt}\|P_tf\|^2_\pi = \frac{d}{dt}\int P_tf(x)^2 \pi(dx).
    $$
    Note that by Proposition \ref{point.wise.limit.generator:1}, point-wise the derivative of $P_tf(x)^2$ with respect to $t$ is $2P_tf(x)LP_tf(x)$. Furthermore, using the same argument as in the Proof of \Cref{weak.generator.domain:0}, we have   $|2P_tf(x)LP_tf(x)| \leq 4\|f\|_{\infty}^2\lambda(x)$, which is $\pi$-integrable using \Cref{prop:well-posed}(b) and recalling that $f \in \mathcal{B}(E)$.  Therefore  we can switch the integral and derivative in the above and we have
    $$
    \frac{d}{dt}\text{Var}_\pi(P_tf) = 2\langle P_tf, LP_t f\rangle_\pi = -2\mathcal{E}(L,P_tf) \leq -2\gamma_L\text{Var}_\pi(P_tf). 
    $$
    Applying the Gr{\"o}nwall inequality and using the semi-group property of $P_t$ then gives the familiar result
    \begin{equation} \label{eq:l2contraction}
        \text{Var}_\pi(P_tf) \leq e^{-2\gamma_L t}\text{Var}_\pi(f)
    \end{equation}
    for $f \in \mathcal{B}(E)$.  Now turning to the total variation distance, using that $\pi(P_tf) = \pi(f)$ and applying the Cauchy--Schwarz inequality shows that
    \begin{align}
        \| \mu P^t - \pi \|_{TV} = \sup_{A \in \mathcal{E}}  \left|\int (P_t\mathbb{I}_A(x) - \pi(A))\left(\frac{d\mu}{d\pi}(x)-1\right)\pi(dx)\right| \leq \chi^2(\mu\|\pi)^{\frac{1}{2}}\sup_{A \in \mathcal{E}} \text{Var}_\pi(P_t\mathbb{I}_A)^{\frac{1}{2}}.
    \end{align}
    Finally applying \eqref{eq:l2contraction} and noting that $\text{Var}_\pi(\mathbb{I}_A) \leq 1/4$ for any $A \in \mathcal{E}$ gives the result.
\end{proof}

In the next two sub-sections we will provide conditions that guarantee existence of a spectral gap for the process. We first treat the case where the balancing function $g$ is bounded before generalising to the unbounded case.
\subsubsection{Bounded $g$ and connections with Metropolis--Hastings}\label{subsubsec:lb_vs_mh}

When $g$ is bounded the locally-balanced Markov jump process can be strongly connected to Metropolis--Hastings algorithms.  We will call a Markov transition kernel of Metropolis--Hastings type if it can be written
$$
P(x,dy) := \alpha(x,y)\gamma(x,dy) + \left( \int (1-\alpha(x,y))\gamma(x,dy) \right) \delta_x(dy),
$$
for some $\alpha$ satisfying $\alpha(x,y)/\alpha(y,x) = t(x,y)$. Note that for any balancing function $g \leq 1$, choosing $\alpha(x,y)=g(t(x,y))$ is a valid choice.  We further observe that the generator of any LBMJP with $g \leq \overline{\lambda}$ for some $\overline{\lambda}<\infty$ can be written
$$
Lf(x) = \bar{\lambda} \int [f(y) - f(x)] \bar{\Gamma}_g(x,dy),
$$
where
$$
\bar{\Gamma}_g(x,dy) = \left( 1 - \int \frac{g\left( t(x,y) \right)}{\overline{\lambda}} \gamma(x,dy) \right)\delta_x(dy) + \frac{g\left( t(x,y) \right)}{\overline{\lambda}} \gamma(x,dy).
$$
Note, therefore, that $\bar{\Gamma}_g$ is a Markov transition kernel of Metropolis--Hastings type with the choice of acceptance rate $\tilde{\alpha}(x,y)=\tilde{g}\left( t(x,y) \right) := g\left(t(x,y) \right)/\overline{\lambda} \leq 1$. The following result shows that any choice of $\tilde{g}$ is not too dissimilar to $\min(1,t)$.

\begin{lemma} \label{lemma:gbound2}
If $g \leq \overline{\lambda}$, $g(1) = 1$ and $g$ is non-decreasing  balancing function then 
$$
\min(1,t) \leq g(t) \leq \overline{\lambda}\min(1,t).
$$
\end{lemma}

\begin{proof}
For the right-hand side inequality, for any balancing function $g(t) \leq \overline{\lambda} \iff g(s)/s \leq \overline{\lambda}$ where $s :=1/t$, meaning $g(s) \leq \overline{\lambda} s$. Since $t$ is arbitrary then so is $s$. Combined with the fact that $g(s) \leq \bar{\lambda}$, this implies the result. The left-hand side inequality follows directly from  \Cref{lem:gbound1}.
\end{proof}

Using \Cref{lemma:gbound2}, an equivalence between the spectral gaps of a LBMJP with bounded $g$ and Metropolis--Hastings algorithms is shown in the below proposition.

\begin{proposition}[Equivalence of spectral gaps between LBMJP and Metropolis--Hastings] \label{prop:spec_equiv}
    If $g(t) \leq \overline{\lambda}$ for some $\overline{\lambda} <\infty$ is non-decreasing balancing function, then the locally-balanced jump process with generator $Lf(x) = \int (f(y) - f(x))g(t)\gamma(x,dy)$ has a positive spectral gap $\gamma_L$ (as in \eqref{our.spectral}) if and only if the Metropolis--Hastings Markov chain with transition kernel $P(x,dy) = \min(1,t)\gamma(x,dy) + \left(\int (1-\min(1,t))\gamma(x,dy)\right)\delta_x(dy)$ does. The same equivalence holds for the spectral gap $\text{G}(P)$, defined in \eqref{general.spectral}.
\end{proposition}

\begin{proof}
Note by \Cref{lemma:gbound2} that for any $f \in L^2(\pi)$
$$
\begin{aligned}
\int (f(y) - f(x))^2 \min\left(1,t(x,y)\right)\pi(dx)\gamma(x,dy) 
&\geq 
\frac{1}{\overline{\lambda}}\int (f(y) - f(x))^2 g\left(t(x,y) \right)\pi(dx)\gamma(x,dy) 
\\
&\geq 
\frac{1}{\overline{\lambda}}\int (f(y) - f(x))^2 \min\left(1,t \right) \pi(dx) \gamma(x,dy).
\end{aligned}
$$
Writing $P(x,dy) = \min(1,t)\gamma(x,dy) + \left(\int (1-\min(1,t))\gamma(x,dy)\right) \delta_x(dy)$, then from the variational characterisation of the spectral gap it directly follows that
$$
\text{Gap}(L) \geq \text{Gap}(P) \geq \frac{\text{Gap}(L)}{\overline{\lambda}}, 
$$
and
$$
\gamma_L \geq \gamma_P \geq \frac{\gamma_L}{\overline{\lambda}}
$$
where $\gamma$ is defined in \eqref{our.spectral}.
\end{proof}

Given the above, we can deduce a stronger equivalence by translating between spectral gaps and exponential ergodicity, as stated below.

\begin{corollary}
The following hold:
\begin{enumerate}
    \item[(i)] Under Assumptions \ref{lambda.positivity:00}-\ref{ass.finite.lambda}, if the Metropolis-Hastings kernel $P$ with proposal $\gamma$, targeting $\pi$ is geometrically ergodic when initialised from $\mu$ with $\chi^2(\mu||\pi)<\infty$, then the LBMJP with the same $\gamma$, bounded $g$ and initialised from $\mu$ is exponentially ergodic.
    \item[(ii)] If two LBMJPs with bounded and non-decreasing $g_1$ and $g_2$ have the same $\gamma$, then either both have positive spectral gap or neither does.
\end{enumerate}
    
\end{corollary}

\begin{proof}
(i) Using Theorem 2.1 of \cite{roberts1997geometric} then if $P$ is geometrically ergodic and $\chi^2(\mu||\pi)<\infty$ then $\|P\|<1$, from which it follows that $\text{Gap}(P) >0$. By \Cref{prop:spec_equiv} we get that $ \text{Gap}(L)>0$, and therefore $\gamma_L>0$. The result follows by  \Cref{prop:geometric_ergodicity}. (ii) This is an immediate consequence of \Cref{prop:spec_equiv} since both locally balanced processes will have a positive spectral gap iff the Metropolis-Hastings algorithm with proposal $\gamma$ has a positive one.   %The argument is almost identical to the proof of \Cref{prop:spec_equiv} since by \Cref{lemma:gbound2} each of $g_1$ and $g_2$ can be upper and lower bounded by some positive multiple of the other. We omit the details.
\end{proof}

\begin{comment}
{\color{blue} Further work that could be added if we choose:}

Some remarks are necessary here:
\begin{itemize}
    \item This should mean that among bounded $g$ the choice $\min(1,t)$ is the best in terms of asymptotic variance, although one must normalise not to $g(1) = 1$ but instead to $\sup_t g(t) = 1$ for this, otherwise the jump rate $\lambda = \sup_t g(t)$ plays a role
    \item We can instead consider $\tilde{\lambda} = \sup_x Z(x)$, and it will sometimes be that $\tilde{\lambda} < \lambda$ depending on $\gamma$ and $\pi$. This may then mean that unbounded g can be considered as a `time-shift' of a bounded $g$ algorithm, by setting $Z(x) \leq \bar{\lambda}(x)$ and then considering mapping between $\bar{\lambda}(x)$ and a constant jump rate process.
\end{itemize}
\end{comment}

\subsubsection{Comparison theorems for unbounded $g$}
\label{subsubsec:compare}

\Cref{prop:geometric_ergodicity} above shows that a positive spectral gap combined with a suitably \emph{warm start} leads to exponential ergodicity in total variation distance. In the next result we show that among non-decreasing choices of $g$ there is a natural ordering on the gaps, which can then be used to greatly simplify the question of exponential ergodicity for a process when $g$ is not bounded.

\begin{proposition}[Comparison of spectral gaps with unbounded $g$]
    Consider two LBMJPs with the same kernel $\gamma$ and  with balancing functions $g_1$ and $g_2$ respectively. Let  $L_1$ and $L_2$ be the generators of the two processes respectively. Assume that $g_1(t) \geq \omega g_2(t)$ for all $t\geq 0$ and some $\omega > 0$. Then $\gamma_{L_1} \geq \omega \cdot \gamma_{L_2}$.  
\end{proposition} 

\begin{proof}
     Note that $\mathcal{B}(E) \subset \mathcal{D}(L_1)\cap\mathcal{D}(L_2)$.  For any $f \in \mathcal{B}(E)$ recall the integral representation
    \begin{align}
    \mathcal{E}(L_1,f) 
    = 
    \frac{1}{2}\int (f(y)-f(x))^2g_1\circ t(x,y)\pi(dx)\gamma(x,dy),
    \end{align}
    from which it is straightforward to see that $\mathcal{E}(L_1,f) \geq \omega \mathcal{E}(L_2,f)$ using that $g_1(t) \geq \omega g_2(t)$ for all $t\geq 0$. 
\end{proof}

The above result, when combined with \Cref{lem:gbound1}, allows us to focus attention wholly on the process with (bounded) balancing function $g(t) = \min(1,t)$.  With this knowledge, it is also immediate that existing results on spectral gaps/geometric ergodicity for Metropolis--Hastings algorithms (e.g.\ \cite{roberts1996geometric,roberts1996exponential,livingstone2019geometric}) can be leveraged in their entirety to establish exponential convergence to equilibrium of a locally-balanced Markov jump process.

\begin{corollary}
    For a given $\gamma(x,\cdot)$, if the LBMJP with $g(t) = \min(1,t)$ has a positive spectral gap, then it will also have a positive spectral gap $\gamma_L$ for all non-decreasing choices of $g$.
\end{corollary} 

\begin{corollary}
    For a given $\gamma(x,dy)$, if a Metropolis--Hastings algorithm with proposal $\gamma(x,dy)$ has a positive spectral gap, then a LBMJP with the same invariant distribution and any choice of non-decreasing $g$ will also have a positive spectral gap $\gamma_L$.
\end{corollary}

\subsection{Uniform Ergodicity}\label{subsec:unif_erg}

The goal of this sub-section is to show that a locally-balanced process can be uniformly ergodic, even on unbounded state spaces, which suggests that their mixing properties can be robust with respect to the starting position. We begin by recalling the following definition.

\begin{definition}[Uniform Ergodicity]\label{unif.ergod.definition:1}
The process $\left(Y_t \right)_{t \geq 0}$ is called uniformly ergodic if there exists $K>0$ and $\rho <1$ such that for all $x \in \textup{supp}(\pi)$
\begin{equation*}
    \| \mathbb{P}_x\left( Y_t \in \cdot \right) - \pi(\cdot) \|_{TV} \leq K \rho^t.
\end{equation*}
\end{definition}

Two important notions for studying ergodicity properties of the process are petite and small sets. We recall (see e.g.\ \cite{meyn.tweedie:09}) that for a discrete time Markov chain $\left( X_n \right)_{n \in \mathbb{N}}$, a set $C$ is called \emph{petite} if there exists a probability measure $\alpha$ on $\mathbb{N}$, a non-trivial measure $\eta$ on $E$ and $\epsilon>0$ such that for all $x \in C$,
    \begin{equation*}
        \int \mathbb{P}_x\left( X_k \in \cdot  \right) \alpha(dk) \geq \epsilon \  \eta(\cdot).
    \end{equation*}
    The set $C$ is called \emph{small} if $\alpha$ can be taken as the Dirac measure $\alpha=\delta_n$ for some $n \in \mathbb{N}$.
    These definitions naturally extend to continuous time processes.

We begin with a result that guarantees that the class of small sets contains all the compact subsets of $E$. This will be helpful later when we establish uniform ergodicity for particular examples.

\begin{proposition}[Compact sets are small]\label{compact.small:0}
    Let Assumptions \ref{lambda.positivity:00}-\ref{regularity.ass:1} hold. Let $\delta >0$ and consider the $\delta$-skeleton of the LBMJP, i.e. the chain $\left( Z_n \right)_{n \in \mathbb{N}}$ with $Z_n=Y_{n \delta}$. There exists a $\delta > 0$ such that for the $\delta$-skeleton all compact sets are small.
\end{proposition}

\begin{proof}[Proof of \Cref{compact.small:0}]
    Let $\delta>0$ and for $n \in \mathbb{N}$ let us write $Q^{\delta}_{n}$ for the $n$-step transition kernel of the $\delta$-skeleton $Z$. Since $\Gamma_g$, is $\pi$-irreducible and $0< \lambda(x)<\infty$, it is easy to see that the LBMJP is also $\pi$-irreducibe. Therefore, the skeleton chain is $\pi$-irreducible, and since $E$ is separable, there exists an open ball $B$ with $\pi(B)>0$. By  \Cref{prop:well-posed}.1, there exists $x \in E$ with $\lambda(x)<\infty$, therefore $Q^{\delta}_{1}\left( x, \left\{ x  \right\} \right)= \mathbb{P}_x\left( Y_{\delta}=x \right)>0$, so the chain is (strongly) aperiodic in the sense of \cite{meyn.tweedie:09}. From \Cref{Feller.proposition:1} the chain is weak Feller. All Assumptions of Theorem 3.4 of \cite{meyn.tweedie:93:1} are satisfied, therefore all compact sets are petite for the skeleton chain. Finally, any such petite set is small as from Theorem 5.5.7 of \cite{meyn.tweedie:09}.
\end{proof}
    
 For any set $C$, we define $h_C:= \inf\{t \geq 0 | Y_t \in C \}$ to be the hitting time to set $C$. It is well known that the notion of uniform ergodicity is closely connected with the behaviour of hitting times of small sets of the process. For the LBMJP we have the following result.

\begin{proposition}[Uniform ergodicity of a LBMJP] \label{prop:finite_hit}
Let \Cref{ass.finite.lambda} hold. Assume further that for any petite set $C$ there exist $M, \lambda^{\ast}>0$ such that for any $x \in E$,  $\mathbb{E}_x[h_C] < M$, and for any $y \in C$, $\lambda(y) \leq \lambda^*$. Then the LBMJP $Y$ is uniformly ergodic.
\end{proposition}

\begin{proof}
Using Markov's inequality
$\mathbb{E}_x[h_C] < M$ implies that $\mathbb{P}_x(h_C < t) \geq 1-M/t$.  Choose $t$ s.t. $M/t < 1$. Then, writing $f^x_C$ for the law of $h_C$ when the process starts from $x$, it holds that
$$
P_t(x,C) 
\geq \int_0^t \mathbb{P}_x(X_t \in C|h_C = s) f^x_C(d s).
$$
By the definition of $h_C$, $h_C = s$ implies that $X_{\tilde{s}} \in C$ for some $\tilde{s} \in (s,t)$. Therefore $\mathbb{P}_x(X_t \in C|h_C = s) \geq e^{-\lambda(X_{\tilde{s}})(t-\tilde{s})} \geq e^{-\lambda^*t}$, which implies that
\begin{equation}
P_t(x,C) 
\geq 
e^{-\lambda^*t}\int_0^t f^x_C(ds)
=
e^{-\lambda^*t}\mathbb{P}_x(h_C < t)
\geq e^{-\lambda^*t}\left(1-\frac{M}{t} \right).
\end{equation}

Since $C $ is petite there exist $\epsilon_0>0$, a probability distribution $\eta$ defined on $E$ and another $\alpha$ defined on $[0,\infty)$ such that 
$\int  P_s(y,A) \alpha(ds) \geq \epsilon_0 \eta (A)$ for all $A\in \mathcal{E}$ and $y \in C$.
Thus for any $x \in E$
\begin{align*}
\int  P_{t+s}(x,A) \alpha(ds) & \geq \int_C \int P_s(y,A) \alpha(ds)  P_t(x,dy)
\geq 
\epsilon_0 \eta (A)P_t(x,C)
\geq
\epsilon\eta (A)\,,
\end{align*}
for $\epsilon=\epsilon_0 e^{-\lambda^*t}(1-M/t)>0$. This means that $E$ is petite.  The process is strongly aperiodic as shown in the proof of  \Cref{compact.small:0}.  Uniform ergodicity of $Y$ follows from applying Theorem 5.2(b) of \cite{down1995exponential}, upon noting that the drift condition $P_TV_T \leq \beta(s)V_T + b\mathbb{I}_E$ for $s \leq T$ with $\beta(s)$ bounded on $[0,T]$ and $\beta(T) <1$ is trivially satisfied by setting $V_T \equiv 1$, $\beta(s) = 0$ for all $s \in [0,T]$ and $b = 1$.
\end{proof}

\subsubsection{Example 1: Simple random walk}

To illustrate how uniform ergodicity can arise for a locally-balanced Markov process when $E$ is unbounded we consider the case $E=\mathbb{N}$, and the locally-balanced process with base kernel 
\begin{equation}\label{srw.kernel:1}
\gamma(x,dy) := \frac{1}{2}(\delta_{x-1}(dy) + \delta_{x+1}(dy)).
\end{equation}

Let us also make the following assumption regarding the tails of $\pi$ and the growth of the function $g$.

\begin{assumption}\label{ass.unif.ergo:1}
   There exist $\tilde{a}, a>0$, $\beta>1$, and $k>0$ such that for all $t \geq 1$, 
   $$
   g(t) \geq t^{\tilde{a}},
   $$ 
and for all $n \in \mathbb{N}$, $\pi(n)>0$, and  for all $n \geq k$,
    $$
    \frac{\pi(n)}{\pi(n+1)} \geq \exp\left\{ a \beta n^{\beta-1} \right\}.
    $$
\end{assumption}

Under this assumption we have the following.

\begin{theorem}\label{theorem.uniform.ergo:00}
    Assume that $\gamma$ is of the form \eqref{srw.kernel:1} and Assumption  \ref{ass.unif.ergo:1} holds. Then the LBMJP is uniformly ergodic.
\end{theorem}

\begin{proof}[Proof of \Cref{theorem.uniform.ergo:00}]
 The proof is presented in \ref{theorem.uniform.ergo:000}.
\end{proof}

As will become evident by the proof of \Cref{theorem.uniform.ergo:00}, the same result holds when the state space is $\left\{ h  n , n \in \mathbb{N} \right\}$ for some $h > 0$ and when $\gamma(x,d y)=2^{-1}\left( \delta_{x-h} + \delta_{x+h} \right)$. Consider a distribution of interest $\pi$ on $\mathbb{R}$ of the form  $\pi(x) \propto \exp\{ -x^{a} \}$
for $a \in (1,2)$. As we will see in \Cref{convergence.diffusion:000}, when $h \rightarrow 0$, after an appropriate rescale of time, the LBMJP converges weakly to an overdamped Langevin diffusion $(S_t)_{t \geq 0}$, with state space $\mathbb{R}$, that solves the stochastic differential equation  
\begin{equation}\label{langevin.diffusion:0001}
    d S_t =\frac{1}{2}\nabla \log\pi(S_t) dt + d B_t
\end{equation}
where $B$ is a Brownian motion. 
We observe that $\pi$ will satisfy \Cref{ass.unif.ergo:1}. Therefore, the LBMJP defined on $\left\{ h n , n \in \mathbb{N} \right\}$ will be uniformly ergodic for a target distribution of the form $\pi(n)\propto \exp\{ -n^{a} \}$
for any $a >1$. On the other hand, the Langevin diffusion \eqref{langevin.diffusion:0001} will not be uniformly ergodic for a target distribution of the form $\pi(x) \propto \exp\{ -|x|^{a} \}$ defined on $\mathbb{R}$ unless $a > 2$ (e.g.\ \cite{roberts1996exponential,sandric2025note}), highlighting a qualitative difference in the mixing behaviour between the processes.

\section{Weak Convergence to an overdamped Langevin diffusion}\label{convergence.diffusion:000}

In this section we focus on the state space is $E=\mathbb{R}^d$ and study the behaviour of the process when the size of the jumps decreases and the frequency increases arbitrarily. 
We will show that with an appropriate space and time rescaling, in the limit a locally-balanced Markov jump process converges weakly to the overdamped Langevin diffusion. Scaling limits of this form have been shown for various Metropolis-Hastings algorithms and have been used to analyse their behaviour in high dimensions and optimally tune various algorithmic parameters (e.g.\ \cite{gelman.gilks.roberts:97}). In this section we make the following additional assumption.

\begin{assumption}\label{ass.scaling.limit:1}
Assume that $\pi$ admits a Lebesgue density $\pi \in C^3(\mathbb{R}^d)$, that $\tilde{M} \geq \pi(x)>0$ for some $\tilde{M} < \infty$ and all $x \in \mathbb{R}^d$, that $g \in C^2(\mathbb{R}_{\geq 0})$, and that there exists $ M>0$ such that for all $x \in \mathbb{R}^d$,
\begin{equation}\label{weak.conv.m.smoothness}
-MI_d \preceq \nabla^2 \log \pi(x) \preceq MI_d.
\end{equation}
\end{assumption}

\begin{remark}
Condition \eqref{weak.conv.m.smoothness} is usually called $M$-smoothness of the potential $-\log\pi(x)$, and is common in numerical analysis (e.g.\ \cite{lytras.metrikopoulos:24}). The assumption restricts the eigenvalues of $\nabla^2\log \pi$ to be in $[-M,M]$.
%   The first inequality in \eqref{weak.conv.m.smoothness} is called weak convexity of the potential $-\log\pi(x)$, and the second is called $M$-smoothness.  Both are common in numerical analysis (e.g.\ \cite{lytras.metrikopoulos:24}).  The assumption restricts the eigenvalues of $\nabla^2\log \pi$ to be in $[-M,M]$.
\end{remark}

The main result of this section is the following.

\begin{theorem}[Weak convergence to an overdamped Langevin diffusion]\label{scaling.limit:00}
   Let Assumptions  \ref{lambda.positivity:00}-\ref{ass.finite.lambda} and \ref{ass.scaling.limit:1} hold, and let $\sigma_n \in (0,1)$ be a sequence with $\lim_{n \rightarrow \infty} \sigma_n =0$. For every $n \in \mathbb{N}$, let $\left( Y^{n}_t \right)_{t \geq 0}$ be the $d$-dimensional LBMJP with $\gamma(x,\cdot)$ chosen to be a $N(x,\sigma_n^2I_d)$ distribution and $ \left(S^n_t\right)_{t \geq 0}$ defined as $S^n_t:=Y_{\sigma_n^{-2}t}$. Let $\left( S_t \right)_{t \geq 0}$ be the overdamped Langevin diffusion process governed by the stochastic differential equation
   $$
   dS_t= \frac{1}{2} \nabla \log \pi(S_t)dt+dB_t.
   $$
   Then $S^n \xrightarrow{n \rightarrow \infty}S$ weakly in the Skorokhod topology.
\end{theorem}

\begin{proof}[Proof of \Cref{scaling.limit:00}]
The proof is presented in \ref{appendix:proof.scaling.limit:00}.
\end{proof}

\begin{remark}
 The proof of  \Cref{scaling.limit:00} can be straightforwardly generalised to non-Gaussian $\gamma$ with similarly decaying variance. If $z_n$ is the jump for the $n$th process under the law of $\gamma$ then the main restriction is to control the moments of $e(x,\sigma_n)=\phi(x+z_n)-\phi(x)$, and in particular guarantee that 
    \begin{equation*}
        \sup_{x \in C} \mathbb{E}_{\gamma}\left[ e(x,\sigma_n)^8 \right]
    \end{equation*}
    to control the higher order terms (such as the term $A(x,\sigma_n)$ in the proof)
    In particular, we observe that kernels $\gamma$ such as those considered in \Cref{subsec:unif_erg} satisfy this property.
\end{remark}

\section{Discussion}\label{sec::Discussion}

\subsection{Use in Monte Carlo simulation}\label{sec::Monte Carlo}

There are two natural ways to use locally-balanced Markov processes for Monte Carlo computation of the integral $\int_E f(x)\pi(dx)$. The first is to simulate a realisation of an LBMJP for $T$ units of time, and then compute ergodic averages along the trajectory.  For simplicity we assume that $T := T_N = \sum_{j=1}^N \tau_j$, where $\tau_j \sim \text{Exp}(\lambda(X_{j-1})$ and $\{X_0,X_1,...,X_{N-1})$ is the embedded Markov chain with transition kernel $\Gamma_g$ as in \eqref{def.gamma.g:1}.  This equates to using the estimator
$$
\hat{f}_N^{MC} := \frac{1}{T_N}\int_0^{T_N} f(Y_t)dt = \frac{1}{T_N}\sum_{j=1}^N \tau_j f(X_{j-1}).
$$
Another approach to computing the same integral is to refrain from simulating the continuous-time process in its entirety, and instead only simulate the embedded Markov chain $\{X_0,X_1,...,\\
X_{N-1}\}$ and use the importance weighted estimator
$$
\hat{f}_N^{IS} := \sum_{j=1}^N \frac{\lambda(X_{j-1})^{-1}}{\sum_{k=1}^N \lambda(X_{k-1})^{-1}} f(X_{j-1}).
$$
In fact it can be shown that $\hat{f}_N^{IS}$ always has lower asymptotic variance than $\hat{f}_N^{MC}$ using a Rao--Blackwellisation argument.  Computing $\hat{f}_T^{IS}$ is called importance tempering in \cite{li2023importance}, and in Theorem B1 of that work the authors quantify how much improvement will be made when comparing the two approaches depending on the particular LBMJP and function $f$ when $E$ is finite.  See also \cite{zhou2022rapid} for more details.

Another approach to computing the integral $\int_E f(x) \pi(dx)$ is to use the kernel $\Gamma_g$ defined in equation \eqref{def.gamma.g:1} to generate proposals within a Metropolis--Hastings algorithm \cite{metropolis1953equation,hastings1970monte}.  If the current state of the chain is $\tilde{X}_i$ and $Y \sim \Gamma_g(\tilde{X}_i,\cdot)$ then the next state $\tilde{X}_{i+1}$ is set to be $Y$ with probability $\min(1, \lambda(\tilde{X}_i)/\lambda(Y))$, otherwise $\tilde{X}_{i+1}$ is set to be $\tilde{X}_i$.  Following this $\int_E f(x)\pi(dx)$ is estimated by
$$
\hat{f}_N^{MH} := \frac{1}{N}\sum_{j=1}^N f(\tilde{X}_{j-1}).
$$
In \cite{zanella2020informed} locally-balanced processes are used in this way to construct Monte Carlo estimators.  It is not immediately clear which of $\hat{f}_N^{IS}$ and $\hat{f}_N^{MH}$ is more effective for a given problem, but some useful discussion in this direction is provided in \cite{zhou2022rapid}.

\subsection{Non-reversible extensions}\label{sec::Non-reversible}

There has been recent interest within the Markov chain monte Carlo sampling community in designing algorithms based on non-reversible Markov processes.  The motivation is that such processes can often have desirable mixing properties (e.g.\ \cite{diaconis.holmes.neal:00}).  Here we briefly illustrate how to extend the framework of locally-balanced Markov processes to allow for non-reversible processes to be constructed in a natural way.

Take an arbitrary Markov kernel $\check{P}$ and mapping $T:E\to E$ such that $Q(z,A) := \delta_{T(z)}(A)$, $\mu Q = \mu$ and the corresponding operator $Qf(z):= \int f(z')Q(z,dz')$ is an isometric involution (see Definition 1 of \cite{andrieu2021peskun}).  Then we define the Markov jump kernel
$$
\check{J}(z,dz') = g\left( \frac{\mu(dz')Q\check{P} Q(z',dz)}{\mu(dz)\check{P}(z,dz')} \right)\check{P}(z,dz').
$$
A rigorous treatment of the above Radon-Nikodym derivative is given in \cite{thin2020nonreversible}.  It is straightforward to show that the condition
$$
\mu(dz)\check{J}(z,dz') = \mu(dz')Q\check{P}Q(z',dz),
$$
known as \textit{skew} or \textit{modified} detailed balance (e.g.\ \cite{andrieu2021peskun,thin2020nonreversible}), is satisfied. Similarly the corresponding operator
$$
\check{L}f(z) := \int [f(z'-f(z)]g\left( \frac{\mu(dz')Q\check{P} Q(z',dz)}{\mu(dz)\check{P}(z,dz')} \right)\check{P}(z,dz')
$$
is $(\mu,Q)$ self-adjoint, meaning that for any suitable $f,g$ it holds that $\langle f,Lg\rangle_\mu = \langle QLQf,g\rangle_\mu$.

The above set up can therefore be used to construct non-reversible locally balanced Markov jump processes.   
We leave a thorough exploration of these non-reversible locally-balanced processes to future work.

\section*{Acknowledgements}
The authors would like to thank the Isaac Newton Institute for Mathematical Sciences, Cambridge, for support and hospitality during the programme {\it Stochastic systems for anomalous diffusion}, where work on this paper was undertaken. They would also like to thank Andrea Bertazzi for helpful discussions, and Codina Cotar for inviting them to participate at the INI programme. Part of the research was conducted while GV was a postdoctoral fellow at UCL, under SL.

\section*{Funding}
SL and GV were supported by an EPSRC New Investigator Award (EP/V055380/1). GZ was suported by ERC, through StG “PrSc-HDBayLe” grant (101076564). This work was supported by EPSRC (EP/Z000580/1).

\bibliographystyle{plainnat}
\bibliography{main.bib}

\appendix

\section{Proof of \Cref{point.wise.limit.generator:1}}\label{appendix:pointwise.limit:000}
\begin{proof}[Proof of \Cref{point.wise.limit.generator:1}]
    We will prove the result when $h \rightarrow 0^+$. The case where $h \rightarrow 0^-$ follows similarly. Recall that when the process starts from $x \in E$, $\tau_1 \exp\left(\lambda(x)\right)$ is the first jumping time. Therefore, for $h >0$, on the event $\left\{ \tau_1 > h \right\}$ we have $X_h=x$. At the same time the density of $\tau_1$ is given by 
\begin{equation*}
    f_{\tau_1}(s)\lambda(x) \exp\left\{  -\lambda(x) s \right\}.
\end{equation*}
We then write,
\begin{align}\label{point.wise.limit.generator:1.1}
    &\frac{1}{h}\left( \mathbb{E}_x\left[ f(X_h) \right] - f(x) \right) = \frac{1}{h}\mathbb{E}_x\left[ \left(f(X_h) -f(x) \right) \mathbb{I}_{\tau_1 \leq h} \right] + \frac{1}{h}\mathbb{E}_x\left[ \left(f(X_h) -f(x) \right) \mathbb{I}_{\tau_1 > h} \right] \nonumber \\
    &= \frac{1}{h} \int_{E}  \int_0^h  \mathbb{E}_y\left[ f(X_{h-s})   \right] \lambda(x) \exp\left\{ -\lambda(x) s \right\}ds\Gamma_g(x,dy) - \frac{1}{h}f(x) \mathbb{P}_x\left( \tau_1 > h \right) \nonumber \\
    &=  \int_{E} \frac{1}{h} \int_0^h  \mathbb{E}_y\left[ f(X_{h-s})   \right] \lambda(x) \exp\left\{ -\lambda(x) s \right\}ds\Gamma_g(x,dy)  - f(x) \frac{1}{h} \left( 1- \exp\left\{ -\lambda(x)h  \right\} \right).
\end{align}
A simple calculation shows that
\begin{equation}\label{point.wise.limit.generator:1.2}
    f(x) \frac{1}{h} \left( 1- \exp\left\{ -\lambda(x)h  \right\} \right) \xrightarrow{h \rightarrow 0} \lambda(x) f(x).
\end{equation}
Furthermore, for any $y \in E$, consider the quantity 
\begin{align}\label{point.wise.limit.generator:1.3}
   A(h,y)&:= \frac{1}{h} \int_0^h  \mathbb{E}_y\left[ f(X_{h-s})   \right] \lambda(x) \exp\left\{ -\lambda(x) s \right\}ds =\lambda(x) \frac{1}{h}\int_0^h\mathbb{E}_y\left[ f(X_u) \right] \exp\{ -\lambda(x)(h-u) \}du \nonumber \\
   =&\lambda(x) \exp\{ -\lambda(x) h \} \frac{1}{h} \int_0^h \mathbb{E}_y\left[ f(X_u) \right] \exp\{ \lambda(x)u \} du.
\end{align}

Now, for any $u,h >0$ let $J=\left\{  \text{there exists a jump on the interval } (u,u+h) \right\}$. We calculate
\begin{align*}
    &\left| \mathbb{E}_y\left[ f(X_{u+h}) \right] - \mathbb{E}_y\left[ f(X_{u}) \right] \right| = \left|  \mathbb{E}_y\left[ f(X_{u+h}) - f(X_{u}) \right] \right| = \left|  \mathbb{E}_y\left[ \left( f(X_{u+h}) - f(X_{u}) \right) \mathbb{I}_J \right] \right| \\
    & \leq  \mathbb{E}_y\left[ \left| f(X_{u+h}) - f(X_{u}) \right| \mathbb{I}_J \right] \leq 2 \| f \|_{\infty} \mathbb{P}_y\left( J \right) \xrightarrow{h \rightarrow 0} 0.
\end{align*}
Therefore the function $u \rightarrow \mathbb{E}_y\left[ f(X_{u}) \right]$ is continuous, and the same holds for the function $u \rightarrow \mathbb{E}_y\left[ f(X_u) \right] \exp\{ \lambda(x)u \}$. From the Fundamental Theorem of Calculus and \eqref{point.wise.limit.generator:1.3}, we get for all $y \in E$
\begin{equation*}
    A(h,y) \xrightarrow{h \rightarrow 0} \lambda(x)f(y),
\end{equation*}
and since $|f|$ is bounded, for all $h \in (0,1)$ $|A(h,y)| \leq \| f \|_{\infty} \lambda(x)$. From Bounded Convergence Theorem 
\begin{equation*}
    \int_{E} \frac{1}{h} \int_0^h  \mathbb{E}_y\left[ f(X_{h-s})   \right] \lambda(x) \exp\left\{ -\lambda(x) s \right\}ds\Gamma_g(x,dy) \xrightarrow{h \rightarrow 0} \lambda(x) \int_E f(y)  \Gamma_g(x,dy).
\end{equation*}
Combining this with \eqref{point.wise.limit.generator:1.1} and \eqref{point.wise.limit.generator:1.2}, the result follows.
\end{proof}

\section{Proof of \Cref{theorem.uniform.ergo:00}}\label{theorem.uniform.ergo:000}

In order to prove the result, we will need a series of intermediate propositions and lemmas. We begin by defining a desired quantity.

\begin{definition}
   For $k \in \mathbb{N}$, we write $h_k=\inf_{t \geq 0} \left\{ Y_t=k \right\}$ for the hitting time of $k$.
\end{definition}

Assume that the current state of the process is $n \in \mathbb{N}$. We denote the probabilities associated with the process moving to the right or to the left
$$
p(n) := \mathbb{P}(X_1 = n+1|X_0 = n)=\frac{g\left(\frac{\pi(n+1)}{\pi(n)}\right)}{g\left(\frac{\pi(n-1)}{\pi(n)}\right) + g\left(\frac{\pi(n+1)}{\pi(n)}\right)}, \quad q(n) := 1-p(n)
$$
respectively. Let
$$
a_n := \frac{1}{\lambda(n)q(n)} = \frac{2}{g\left(\frac{\pi(n-1)}{\pi(n)}\right)},
$$
and define the odds associated with moving to the right as
$$
b(n)=\frac{p(n)}{q(n)}.
$$
For a fixed $k \in \mathbb{N}$, and for $n \geq k+1$, we also define  
$$
\gamma_n := 1 + b(n-1) + b(n-1)b(n-2) + ... + b(n-1)b(n-2)b(n-3)...b(k+1).
$$
%and define
%$$
%\delta_n:=b(n)  + b(n)b(n-1) + ... + b(n)...b(k+1) = b(n) \gamma_n,
%$$
where we omit the dependence of the quantity $\gamma$ on $k$ for ease of notation.

The following lemma relates the hitting times of the process with the above defined quantities.

\begin{lemma}\label{lemma.hitt.from.N:1}
    For all $N  > k \in \mathbb{N}$, assume that the LBMJP starts from $N$. Then, we have 
    \begin{equation} \label{eq:hitt.from.N}
        \mathbb{E}_{N}\left[ h_{k} \right]= \sum_{n=k}^{N-1}a_{n+1} \gamma_{n+1} +\mathbb{E}_{N+1}\left[  h_{N}\right] b(N)\gamma_N
    \end{equation}
\end{lemma}

\begin{proof}[Proof of \Cref{lemma.hitt.from.N:1}]
We observe that 
\begin{equation}\label{lemma.hitt.from.N:0.5}
    \mathbb{E}_N\left[ h_{k} \right]=\sum_{n=k}^{N-1}\mathbb{E}_{n+1}\left[ h_{n} \right].
\end{equation}

Let us write $X^n$ to denote the Markov chain with transition kernel $\Gamma_g$, starting from $n$. We also write $\tau_m \sim \exp\left( \lambda(m) \right)$ to denote the time the process $Y$ stays in $m \in \mathbb{N}$ before jumping. Note that

\begin{align}\label{lemma.hitt.from.N:2}
\mathbb{E}_{n}[h_{n-1}] 
&= 
\mathbb{E}_n\left[ \mathbb{E} \left[ h_{n-1} | X^n \right] \right] =
\mathbb{E}_n\left[ q(n)\tau_{n} + p(n) \left( \tau_n + \mathbb{E}_{n+1}\left[ h_{n-1} | X^{n+1} \right] \right) \right] 
\\
&= 
q(n) \frac{1}{\lambda(n)} + p(n) \frac{1}{\lambda(n)} + p(n) \mathbb{E}_{n+1}\left[  h_{n-1}\right] =
\frac{1}{\lambda(n)} + p(n) \mathbb{E}_{n+1}\left[  h_{n-1}\right]. \nonumber
\end{align}

Now, from the Markov property, we have $\mathbb{E}_{n+1}\left[ h_{n-1} \right]=\mathbb{E}_{n+1}\left[ h_{n} \right]+\mathbb{E}_{n}\left[ h_{n-1} \right]$, so (\ref{lemma.hitt.from.N:2}) becomes

\begin{equation*}
    \mathbb{E}_{n}[h_{n-1}] = \frac{1}{\lambda(n)} + p(n) \mathbb{E}_{n+1}\left[  h_{n}\right] + p(n) \mathbb{E}_{n}\left[  h_{n-1}\right].
\end{equation*}

Rearranging, we get
\begin{equation}\label{lemma.hitt.from.N:3}
 \mathbb{E}_{n}[h_{n-1}] = \frac{1}{\lambda(n)q(n)} + \frac{p(n)}{q(n)} \mathbb{E}_{n+1}\left[  h_{n}\right] = a(n) + b(n) \mathbb{E}_{n+1}\left[  h_{n}\right] .
\end{equation}

We therefore have that $\mathbb{E}_{n}[h_{n-1}]= a(n)+a(n+1)b(n) + b(n+1) b(n) \mathbb{E}_{n+2}[h_{n+1}]$.  Applying this formula recursively gives
\begin{align}\label{lemma.hitt.from.N:4}
    \mathbb{E}_{n}[h_{n-1}]= &a(n) + a(n+1) b(n) + a(n+2) b(n+1) b(n) + \dots \nonumber\\
    &+ a(N) b(N-1) b(N-2) \dots b(n) +b(N) b(N-1) \dots b(n) \mathbb{E}_{N+1}\left[ h_N \right]
\end{align}
Using (\ref{lemma.hitt.from.N:0.5}) and on summing over $n$, we recover \eqref{eq:hitt.from.N}.

\end{proof}

An interesting corollary is the following. It shows that the behaviour of the series of $a_n$ is crucial for the process to come down from infinity in finite time, which itself is crucial for uniform ergodicity.

\begin{corollary}\label{non.uniform.ergo.result:1}
    If $\sum_{n=1}^{\infty}a_n = +\infty$, then for any $k$, $\lim_{N \rightarrow +\infty} \mathbb{E}_{N}\left[ h_k \right] = + \infty$.
\end{corollary}

\begin{proof}[Proof of \Cref{non.uniform.ergo.result:1}]
    Since $\gamma(n) \geq 1$, from \eqref{eq:hitt.from.N} we get for all $N > k$
    $$
    \mathbb{E}_{N}\left[ h_k \right] \geq \sum_{n=k}^{N-1}a_n,
    $$
    and the result follows on letting $N \rightarrow +\infty$.
\end{proof}

We will now make the following assumption and state any further results in this subsection assuming that it holds. As we will see later, this assumption is weaker than Assumption \ref{ass.unif.ergo:1}.

\begin{assumption}\label{ass.unif.ergo:0}
    Assume that there exists  $k \in \mathbb{N}$ and $p=p^{(k)}$ and $\lambda=\lambda^{(k)}$ such that for all $n \geq k$ $p(n) \leq p<\frac{1}{2}$ and $\lambda(n) \geq \lambda > 0$. Assume further that 
    \begin{equation}\label{lemma.hitt.from.N:6}
   n b(n) \xrightarrow{n \rightarrow \infty} 0.
    \end{equation}
\end{assumption}

We then have the following.

\begin{proposition}\label{lemma.ass:1} Assume that Assumption \ref{ass.unif.ergo:0} holds. Then for any compact set $C \subset \mathbb{N}$,
\begin{equation}\label{come.down.infty:1}
      \limsup_{N \rightarrow \infty} \mathbb{E}_{N}\left[ h_{C} \right] < \infty \iff \sum_{n=1}^{\infty}a(n)<\infty.
   \end{equation}
\end{proposition}

\begin{proof}[Proof of  \Cref{lemma.ass:1}]
Let $C \subset \mathbb{N}$ be a compact set and let $k=\sup C$. We first observe that for any $N > k$, for the LBMJP with $\gamma$ as in \eqref{srw.kernel:1}, $h_C=h_k$. We will therefore prove \eqref{come.down.infty:1} with $h_k$ instead of $h_C$.

First of all, the LHS of \eqref{come.down.infty:1} implies the RHS due to \Cref{non.uniform.ergo.result:1}.

Let us assume that the RHS of \eqref{come.down.infty:1} holds. Let $\lambda >0$ and $p \in (0,1/2)$ be as in Assumption \ref{ass.unif.ergo:0}. Let $\tilde{Y}$ be a Markov jump process defined on $\mathbb{N}$ with constant jump rate $\lambda$ and with jump law given by $\tilde{\Gamma}(x,dy)=p \ \delta_{x+1}+ \left( 1-p \right) \delta_{x-1}$. Let $\tilde{h}_n=\inf_{ t \geq 0} \left\{ \tilde{Y}_t=n \right\}$, and let $\tilde{X}$ be the discrete time Random Walk on $\mathbb{N}$ with constant probability of jumping to the right be $p$. 

We can couple the processes $\left(Y_t\right)_{t \geq 0}$ and $\left(\tilde{Y}_t\right)_{t \geq 0}$ by using the same $\exp(1)$ random variables to generate their jumping times and the same uniform distributions to generate their jumps. A simple comparison between the two coupled processes shows that for any $n \geq k$
   \begin{equation*}
       \mathbb{E}_{n+1}\left[ h_n \right] \leq \mathbb{E}_{n+1}\left[ \tilde{h}_n \right],
   \end{equation*}
   
   and therefore
   \begin{equation}\label{one.down.expectation.bounded:1}
       \mathbb{E}_{n+1}\left[ h_n \right] \leq \mathbb{E}_{n+1}\left[ \tilde{h}_n \right] \leq \frac{1}{\lambda} \mathbb{E}_{n+1}\left[ \inf_{m \in \mathbb{N}}\left\{ \tilde{X}_m=n  \right\} \right] = \frac{1}{\left(1-2p\right)\lambda},
   \end{equation}
   
where for the final equation we use standard results on the hitting times of random walk (e.g.\ \cite{levin.peres.wilmer:06}, above equation (2.14)). Furthermore, we observe that for all $n \geq k$,
$$
b(n)<1.
$$
Therefore, using \eqref{lemma.hitt.from.N:6}, we get that for $n \geq k$
\begin{align}\label{lemma.hitt.from.N:7}
b(n)\gamma_n&=b(n)+b(n)b(n-1)+\dots + b(n)b(n-1)\dots b(k+1) \nonumber \\
    &\leq (n-k) b(n) \leq n b(n) \xrightarrow{n \rightarrow +\infty} 0. 
\end{align}
Combining this with \eqref{one.down.expectation.bounded:1}, we get
\begin{equation}\label{lemma.hitt.from.N:5}
       \lim_{N \rightarrow \infty}b(N)\gamma_N \mathbb{E}_{N+1}\left[ h_N \right] = 0.
   \end{equation}
Finally, since for all $n \geq k$,$b(n)<1$  and $b(n)\xrightarrow{n \rightarrow \infty}0$, we have $\gamma(n) \xrightarrow{n \rightarrow \infty} 1$. Since $\sum_{n=1}^{\infty}a(n)<\infty$, we get
\begin{equation}\label{gamma.to.one:2}
    \sum_{n=k}^{\infty}a(n+1)\gamma(n+1) < \infty. 
    \end{equation}
    This, along with  \eqref{lemma.hitt.from.N:5} and \Cref{lemma.hitt.from.N:1} shows that 
    \begin{equation*}
        \limsup_{N \rightarrow \infty}\mathbb{E}_N\left[ h_k \right] < \infty.
    \end{equation*}
    This completes the proof.
\end{proof}

\begin{proof}[Proof of \Cref{theorem.uniform.ergo:00}]

From Assumption \ref{ass.unif.ergo:1} we get that there exists $ k \in \mathbb{N}$ such that for all $n > k$
\begin{equation*}
    \frac{\pi(n+1)}{\pi(n)} \leq 1 , \ \ \ \frac{\pi(n-1)}{\pi(n)} \geq \exp\left\{ a \beta (n-1)^{\beta-1} \right\}
\end{equation*}
and since $g$ is increasing, $g(1)=1$ and $g(t) \geq t^{\tilde{a}}$ for $t \geq 1$, we get
\begin{equation}\label{b.upper.bound:392}
     b(n)=\frac{g\left( \frac{\pi(n+1)}{\pi(n)} \right)}{g\left( \frac{\pi(n-1)}{\pi(n)} \right)} \leq \exp\left\{ -a\tilde{a}\beta (n-1)^{\beta-1} \right\},
\end{equation}
and therefore
\begin{equation}
\lim_{n \rightarrow +\infty} n b(n)=0.
\end{equation}
Furthermore, for all $n > k$
\begin{equation*}
\lambda(n)=g\left(\frac{\pi(n-1)}{\pi(n)}\right) + g\left(\frac{\pi(n+1)}{\pi(n)}\right) \geq \left(\frac{\pi(n-1)}{\pi(n)}\right)^{\tilde{a}} \geq \exp\left\{ \tilde{a} a \beta n^{\beta-1} \right\} \geq \exp\left\{ \tilde{a} a \beta k^{\beta-1} \right\}
\end{equation*}
and due to \eqref{b.upper.bound:392}
\begin{equation*}
p(n)=\frac{g\left(\frac{\pi(n+1)}{\pi(n)}\right)}{g\left(\frac{\pi(n-1)}{\pi(n)}\right) + g\left(\frac{\pi(n+1)}{\pi(n)}\right)} \leq \frac{g\left( \frac{\pi(n+1)}{\pi(n)} \right)}{g\left( \frac{\pi(n-1)}{\pi(n)} \right)} \leq \exp\left\{ -a\tilde{a}\beta k^{\beta-1} \right\} <\frac{1}{2},
\end{equation*}
where the last inequality holds by choosing $k$ sufficiently large. Therefore, Assumption 
\ref{ass.unif.ergo:0} holds. 

Furthermore, we observe that for $n > k$ 
$$
a(n)=\frac{2}{g\left( \frac{\pi(n-1)}{\pi(n)} \right)} \leq 2 \exp\left\{ -\tilde{a}a\beta\left( n-1 \right)^{\beta-1} \right\},
$$
and therefore $\sum_{n=1}^{\infty}a(n)<\infty$.

By \Cref{lemma.ass:1}, for any compact set $C$, there exists $M>0$ such that $ \mathbb{E}_N\left[ h_C \right] \leq M$ for all $N \in \mathbb{N}$. Note that by Assumption \ref{ass.unif.ergo:1} and the form of $\gamma$, the conditions of \Cref{compact.small:0} are satisfied, so every compact set $C$ is small. Therefore, all assumptions of  \Cref{prop:finite_hit} are satisfied and the result follows.
\end{proof}

\section{Proof of \Cref{scaling.limit:00}}\label{appendix:proof.scaling.limit:00}

\begin{proof}
Let $t(x,y)=\pi(y)/\pi(x)$ and let $b(x)=\log g(\exp\{ x \})$ so that $g(t)= \exp\left\{ b\left( \log t \right) \right\}$. For convenience, let us write
$\phi=\log f \in C^3$. Let $L^n$ and $L$ be the weak generators of $S^n$ and $S$ respectively. Recall from \Cref{weak.generator.domain:0} that for $f \in C^{\infty}_c(\mathbb{R}^d)$,
\begin{align*}
L^n f(x) 
&= \sigma_n^{-2} \mathbb{E}_{Y \sim N(x,\sigma_n^2I_d)}\left[ \left( f(Y)-f(x) \right) g \circ t(x,Y) \right] 
\\
&=
\sigma_n^{-1} \mathbb{E}_{Z \sim N(0,I_d)}\left[ \left( f(x+\sigma_n Z)-f(x) \right) \exp \left\{ b \left( \phi(x+\sigma_n Z)- \phi(x) \right) \right\} \right],
\end{align*}
while it is well known (see e.g.\ \cite{gelman.gilks.roberts:97}) that
\begin{equation*}
    Lf(x)=\frac{1}{2}\nabla \phi(x) \nabla f(x) + \frac{1}{2} \Delta f(x).
\end{equation*}
 Let $f \in C_c^{\infty}$ and $C$ a compact set. 
From $g(t)=tg(1/t)$ we get $g'(t)=g(1/t)-t^{-1} g'(1/t)$ so $g'(1)=g(1)/2=1/2$. Therefore $b'(0)=1/2$. Also $b(0)=0$. Noting that $\phi$ is continuous, a Taylor series expansion of $b$ around $0$ shows that $\exists \sigma_n^{(1)} \in (0,\sigma_n)$ such that
\begin{equation}\label{scaling.limit.sigma.1:1}
\exp \left\{ b \left( \phi(x+\sigma_n Z)-  \phi(x) \right) \right\} = \exp \left\{ \frac{1}{2} e(x,\sigma_n)  + \frac{1}{2}b''(e(x, \sigma_n^{(1)})) e(x,\sigma_n)^2 \right\} 
\end{equation}
for some $e(x,\sigma_n):=\phi(x+\sigma_n Z) - \phi(x)$.
Using another Taylor series expansion, we observe that for any $\sigma_n$, $\exists \sigma_n^{(3)} \in (0, \sigma_n)$ such that 
\begin{equation}\label{taylor.for.error:1}
e(x,\sigma_n)= \sigma_n \nabla \phi (x)\cdot Z + \frac{1}{2}\sigma_n^2 Z^T \nabla^2 \phi (x+ \sigma_n^{(3)} Z ) Z.
\end{equation}
 Therefore, the facts that $-M I_d \leq \nabla^2 \phi (x) \leq M I_d$ and $\| \nabla \phi \|_{\infty,C}=\sup_{x \in C}|\nabla \phi (x)|< \infty$ imply that for any $Z$
 \begin{equation*}
 \sup_{x \in C} |e(x,\sigma_n)| \leq \sigma_n  \| \nabla \phi \|_{\infty,C} \|  Z \|_2 + \frac{1}{2} \sigma_n^2 M \| Z \|_2^2,
 \end{equation*}
 which in turn implies that for all $k>0$ $\exists M_k>0$ such that
 \begin{equation}\label{bound.phi.error:1}
 \mathbb{E}_{Z \sim N(0,I_d)}\left[ \sup_{x \in C} |e(x,\sigma_n)|^k \right] \leq M_k \sigma_n^k.
 \end{equation}
Now $\exp\{ u \}=1+u + \exp\{ \xi \} u^2/2$, for some $\xi \in (-u,u)$. Therefore, using (\ref{scaling.limit.sigma.1:1}) we get that there exists a $\xi \in \left( \min \left\{ 0, b(e(x,\sigma_n))  \right\}, \max \left\{ 0, b(e(x,\sigma_n))  \right\} \right)$ such that
\begin{align*}
&\exp \left\{ b \left( \phi(x+\sigma_n Z)-  \phi(x) \right) \right\}= \exp \left\{ \frac{1}{2} e(x,\sigma_n)  + \frac{1}{2}b''(e(x, \sigma_n^{(1)})) e(x,\sigma_n)^2 \right\} \\
&=1+ \frac{1}{2} e(x,\sigma_n)  + \frac{1}{2}b''(e(x, \sigma_n^{(1)})) e(x,\sigma_n)^2 + \frac{1}{2} \exp \left\{ \xi \right\} \left( \frac{1}{2} e(x,\sigma_n)  + \frac{1}{2}b''(e(x, \sigma_n^{(1)})) e(x,\sigma_n)^2 \right)^2 \\
&= 1+ \frac{1}{2} e(x,\sigma_n)  + \left( \frac{1}{2}b''(e(x, \sigma_n^{(1)})) + \frac{1}{8} \exp\{ \xi \} \right) e(x,\sigma_n)^2 +\frac{1}{4}\exp\{ \xi \} b''(e(x,\sigma_n^{(1)})) e(x,\sigma_n)^3 \\
&+ \frac{1}{8} \exp\{ \xi \} \left[ b''(e(x,\sigma_n^{(1)})) \right]^2 e(x,\sigma_n)^4.
\end{align*}
At the same time, from a Taylor expansion up to second degree, $\exists \sigma_n^{(2)} \in (0, \sigma_n)$ such that  for any $Z \in \mathbb{R}^d$
$$
f(x+\sigma_n Z) -f(x)
= 
\sigma_n \nabla f(x) \cdot Z + \frac{1}{2}\sigma_n^2 Z^T \nabla^2f(x)Z + \frac{1}{6} \sigma_n^3\sum_{i,j,k=1}^d\partial_i \partial_j \partial_k f(x+\sigma_n^{(2)}Z) Z_i Z_j Z_k.
$$
Overall this gives

\begin{align}\label{scaling.limit.huge.sum:1}
( f(x+\sigma_n Z)& - f(x) ) \exp\left\{ b \left( \phi(x+\sigma_n Z) -\phi(x) \right) \right\} = \\
&\sigma_n \nabla f(x) \cdot Z+\frac{1}{2}\sigma_n^2 Z^T \nabla^2f(x)Z + \frac{1}{6} \sigma_n^3\sum_{i,j,k=1}^d\partial_i \partial_j \partial_k f(x+\sigma_n^{(2)}Z) Z_i Z_j Z_k \nonumber \\
&+ \frac{1}{2} \sigma_n e(x,\sigma_n) \nabla f(x) \cdot Z +\frac{1}{4} \sigma_n^2 e(x,\sigma_n) Z^T \nabla^2f(x)Z \nonumber \\
&+ \frac{1}{12} \sigma_n^3 e(x,\sigma_n) \sum_{i,j,k=1}^d\partial_i \partial_j \partial_k f(x+\sigma_n^{(2)}Z) Z_i Z_j Z_k \nonumber \\
&+ \sigma_n e(x,\sigma_n)^2 \left( \frac{1}{2}b''(e(x, \sigma_n^{(1)})) + \frac{1}{8} \exp\{ \xi \} \right) \nabla f(x) \cdot Z \nonumber \\
&+ \frac{1}{2}\sigma_n^2 e(x,\sigma_n)^2 \left( \frac{1}{2}b''(e(x, \sigma_n^{(1)})) + \frac{1}{8} \exp\{ \xi \} \right) Z^T \nabla^2f(x)Z \nonumber \\
&+ \frac{1}{6} \sigma_n^3 e(x,\sigma_n)^2 \left( \frac{1}{2}b''(e(x, \sigma_n^1)) + \frac{1}{8} \exp\{ \xi \} \right) \sum_{i,j,k=1}^d\partial_i \partial_j \partial_k f(x+\sigma_n^{(2)}Z) Z_i Z_j Z_k \nonumber \\
&+ \frac{1}{4} \sigma_n e(x,\sigma_n)^3 \exp\{ \xi \} b''(e(x,\sigma_n^{(1)})) \nabla f(x) \cdot Z \nonumber \\
&+ \frac{1}{8} \sigma_n^2 e(x,\sigma_n)^3\exp\{ \xi \} b''(e(x,\sigma_n^{(1)})) Z^T \nabla^2f(x)Z \nonumber \\
&+ \frac{1}{24} \sigma_n^3 e(x,\sigma_n)^3 \exp\{ \xi \} b''(e(x,\sigma_n^{(1)})) \sum_{i,j,k=1}^d\partial_i \partial_j \partial_k f(x+\sigma_n^{(2)}Z) Z_i Z_j Z_k \nonumber \\
&+ \frac{1}{8} \sigma_n e(x,\sigma_n)^4 \exp\{ \xi \} \left[ b''(e(x,\sigma_n^{(1)})) \right]^2 \nabla f(x) \cdot Z \nonumber \\
&+ \frac{1}{16} \sigma_n^2 e(x,\sigma_n)^4 \exp\{ \xi \} \left[ b''(e(x,\sigma_n^{(1)})) \right]^2 Z^T \nabla^2f(x)Z \nonumber \\
&+ \frac{1}{48} \sigma_n^3 e(x,\sigma_n)^4 \exp\{ \xi \} \left[ b''(e(x,\sigma_n^{(1)})) \right]^2 \sum_{i,j,k=1}^d\partial_i \partial_j \partial_k f(x+\sigma_n^{(2)}Z) Z_i Z_j Z_k. \nonumber
\end{align}
Noting that $\mathbb{E}_{Z \sim N(0,I_d)}\left[ \nabla f(x) \cdot Z \right]=0$ and that $\mathbb{E}_{Z \in N(0,I_d)}\left[ Z^T \nabla^2 f(x) Z \right]=\Delta f(x)$ we get that
\begin{align}\label{scaling.limit.middle.equation:1}
&L^n(x)=\sigma_n^{-2}\mathbb{E}_{Z \sim N(0,I_d)}\left[ \left( f(x+\sigma_n Z)-f(x) \right) \exp \left\{ b \left( \phi(x+\sigma_n Z)- \phi(x) \right) \right\} \right]\\
\nonumber &=\frac{1}{2} \Delta f(x)  + \frac{1}{2} \sigma_n^{-1} \mathbb{E}_{Z \sim N(0,I_d)}\left[ e(x,\sigma_n) \nabla f(x) \cdot Z \right] +\sigma_n^{-2}A(x,\sigma_n),
\end{align}
where $A(x,\sigma_n)$ is expectation of the remaining terms of the sum in (\ref{scaling.limit.huge.sum:1}) (i.e. except for first, second and fourth term). 

We will now control the term $A(x,\sigma_n)$ and prove it to be of order $\sigma_n^3$. First of all, recall that $\xi \leq \max\{ 0 , b(e(x,\sigma_n)) \}$. Furthermore, from  \Cref{lem:gbound1} we have $\min \{ 1 , t \} \leq g(t) \leq \max \{ 1, t \}$, therefore for any $k \in \mathbb{N}$,
\begin{align}\label{scaling.limit.bound.xi:1}
\exp\{ k \xi \} &\leq \exp\left\{ \max \left\{ 0 , k b(e(x,\sigma_n))  \right\} \right\} = \exp \left\{ \max \left\{ 0, k \log g(\exp \{ e(x,\sigma_n) \} ) \right\} \right\}\\ \nonumber
&= \max \left\{ 1, \left( g\left( \exp\{ e(x,\sigma_n) \} \right) \right)^k \right\} \leq \max \left\{ 1, \left( \exp\{ e(x,\sigma_n) \} \right)^k  \right\} \\ \nonumber
&= \max \left\{ 1, \left( \frac{\pi(x+\sigma_n Z)}{\pi(x)} \right)^k  \right\} \leq \max \left\{ 1 , \frac{\tilde{M}^k}{\pi(x)^k} \right\},
\end{align}
which is bounded uniformly on $x$ on the compact set $C$ since $\pi \in C^0$ and $\pi(x)>0$ for all $x \in \mathbb{R}^d$.

At the same time, in order to control $A(x,\sigma_n)$ we need to control the term $b''(e(x,\sigma_n^{(1)}))$ appearing in the sum.
In order to do this, note that from (\ref{scaling.limit.sigma.1:1}) we have that
\begin{equation*}
b''(e(x,\sigma_n^{(1)}))=2\frac{b(e(x,\sigma_n))-\frac{1}{2}e(x,\sigma_n)}{e(x,\sigma_n)^2}
\end{equation*}
when $x, Z \in \mathbb{R}^d$ such that $e(x,\sigma_n) \neq 0$ and we can define $b''(e(x,\sigma_n^{(1)}))=b''(0)$ when $e(x,\sigma_n)=0$. 
Let $z(e)=2\frac{b(e)-\frac{1}{2}e}{e^2}$. Since $b(0)=0$ and $b'(0)=\frac{1}{2}$, from L'Hospital we get
\begin{equation*}
\lim_{e \rightarrow 0}z(e)=b''(0)=\frac{1}{2}+g''(1)-g''(1)^2.
\end{equation*}
Therefore, there exists $e_0$ such that for any $x \in C$ and $Z \in \mathbb{R}^d$ with $e(x,\sigma_n)<e_0$,
\begin{equation*}
\left| 2\frac{b(e(x,\sigma_n))-\frac{1}{2}e(x,\sigma_n)}{e(x,\sigma_n)^2} \right| \leq 1+|g''(1)|+g''(1)^2.
\end{equation*}
On the other hand, using a similar argument as in (\ref{scaling.limit.bound.xi:1}) we get that 
\begin{equation*}
\sup_{x \in C}|b(e(x,\sigma))|<\infty,
\end{equation*}
which implies that 
\begin{equation*}
\sup_{x \in C, Z \in \mathbb{R}^d, e(x,\sigma_n) \geq e_0}\left| 2 \frac{b(e(x,\sigma_n))-\frac{1}{2}e(x,\sigma_n)}{e(x,\sigma_n)^2} \right| <\infty .
\end{equation*}
Therefore, 
\begin{equation}\label{scaling.limit.b2:1}
\sup_{x \in C, Z \in \mathbb{R}^d}b''(e(x,\sigma_n^{(1)}))=\sup_{x \in C, Z \in \mathbb{R}^d} \left| 2\frac{b(e(x,\sigma_n))-\frac{1}{2}e(x,\sigma_n)}{e(x,\sigma_n)^2} \right| < \infty.
\end{equation}

Using the fact that $f \in C^{\infty}_c$, that $Z \sim N(0,I_d)$, that for all $n \in \mathbb{N}$ $e(x,\sigma_n)$ has all the moments finite (due to (\ref{bound.phi.error:1})), using (\ref{scaling.limit.b2:1}) and (\ref{bound.phi.error:1}), and by carefully considering all the terms of the sum in $A(x,\sigma_n)$, we get that
\begin{equation}\label{scaling.limit.bound.on.A:1}
\limsup_{n \rightarrow \infty}\frac{\sup_{x \in C}|A(x,\sigma_n)|}{\sigma_n^3}<\infty. 
\end{equation}
Furthermore, recall that due to (\ref{taylor.for.error:1}) we have
\begin{align*}
&e(x,\sigma_n) \nabla f(x) \cdot Z = \sigma_n \left( \nabla f(x) \cdot Z \right)  \left( \nabla \phi (x) \cdot Z \right) + \frac{1}{2}  \sigma_n^2 \left( Z^T\nabla^2 \phi(x+\sigma_n^{(3)}Z) Z \right) \left( \nabla f (x) \cdot Z  \right) \\
&= \sigma_n \left( \sum_{i,j=1}^d \partial_i f(x) \partial_j \phi(x) Z_i Z_j \right)  + \frac{1}{2}   \sigma_n^2 \sum_{i,j,k=1}^d \partial_k f(x) \partial_i \partial_j \phi(x+\sigma_n^{(3)}Z) Z_i Z_j Z_k
\end{align*}
and therefore
 \begin{align}\label{scaling.limit.hamiltonian.part:1}
 \mathbb{E}_{Z \sim N(0,I_d)}\left[ e(x,\sigma_n) \nabla f(x) \cdot Z \right]
 = &\sigma_n \nabla \phi(x) \cdot \nabla f(x)
 \\
 &+ \sigma_n^2 \mathbb{E}_{Z \sim N(0,I_d)}\left[ \frac{1}{2} \sum_{i,j,k=1}^d \partial_k f(x)  \partial_i \partial_j \phi(x+\sigma_n^{(3)}Z) Z_i Z_j Z_k \right],
 \end{align}
and using again the fact that for all $y \in \mathbb{R}^d$, $-M I_d \preceq \nabla^2\phi(y) \preceq MI_d$ we get that
\begin{equation}\label{scaling.limit.hamiltonian.part:2}
 \sup_{x \in C} \mathbb{E}_{Z \sim N(0,I_d)}\left[ \frac{1}{2} \sum_{i,j,k=1}^d \partial_k f(x)  \partial_i \partial_j \phi(x+\sigma_n^{(3)}Z) Z_i Z_j Z_k \right] <\infty
\end{equation}
Overall using (\ref{scaling.limit.middle.equation:1}),  (\ref{scaling.limit.bound.on.A:1}), (\ref{scaling.limit.hamiltonian.part:1}) and (\ref{scaling.limit.hamiltonian.part:2}), we get
\begin{equation*}
L^nf(x)=\left(  \frac{1}{2} \nabla \phi(x) \cdot \nabla f(x) + \frac{1}{2} \Delta f(x)  \right) + \sigma_n^{-2}B(x,\sigma_n^3)
\end{equation*}
with
\begin{equation*}
\limsup_{n \rightarrow \infty}\frac{\sup_{x \in C}|B(x,\sigma_n)|}{\sigma_n^3}<\infty. 
\end{equation*}
For all $n$, the Martingale problem $(\mu,L^n,C^{\infty}_c)$ is solved by the jump process, due to \Cref{weak.generator.domain:0}.

All assumptions of Theorem 8.1 of \cite{monmarche.rousset.zitt:22} are satisfied and the result follows.
\end{proof}

\maketitle

\end{document}